\documentclass{amsart}

\newtheorem{thm}{Theorem}[section]

\newtheorem{prop}{Proposition}[section]
\newtheorem{coro}{Corollary}[section]

\newtheorem{remark}{Remark}[section]

\newtheorem{alphthm}{Theorem}[section]

\newtheorem{alphcoro}{Corollary}[section]

\newtheorem{alphprop}{Proposition}[section]

\newtheorem{alphlemma}{Lemma}[section]

\newcommand{\lct}{\; \raisebox{-.96ex}{$\stackrel{\textstyle <}{\sim}$} \;}
\newcommand{\gct}{\; \raisebox{-.96ex}{$\stackrel{\textstyle >}{\sim}$} \;}
\newcommand{\mbb}{\mathbb}

\begin{document}

\title{Fourier restriction in low fractal dimensions}

\author{Bassam Shayya}
\address{Department of Mathematics\\
         American University of Beirut\\
         Beirut\\
         Lebanon}
\email{bshayya@aub.edu.lb}



\begin{abstract}
Let $S \subset \mbb R^n$ be a smooth compact hypersurface with a strictly 
positive second fundamental form, $E$ be the Fourier extension operator on 
$S$, and $X$ be a Lebesgue measurable subset of $\mbb R^n$. If $X$ contains
a ball of each radius, then the problem of determining the range of 
exponents $(p,q)$ for which the estimate 
$\| Ef \|_{L^q(X)} \lct \| f \|_{L^p(S)}$ holds is equivalent to the
restriction conjecture. In this paper, we study the estimate under the
following assumption on the set $X$: there is a number $0 < \alpha \leq n$ 
such that $|X \cap B_R| \lct R^\alpha$ for all balls $B_R$ in $\mbb R^n$ of 
radius $R \geq 1$. On the left-hand side of this estimate, we are 
integrating the function $|Ef(x)|^q$ against the measure $\chi_X dx$. Our 
approach consists of replacing the characteristic function $\chi_X$ of $X$ 
by an appropriate weight function $H$, and studying the resulting estimate 
in three different regimes: small values of $\alpha$, intermediate values 
of $\alpha$, and large values of $\alpha$. In the first regime, we 
establish the estimate by using already available methods. In the second 
regime, we prove a weighted H\"{o}lder-type inequality that holds for 
general non-negative Lebesgue measurable functions on $\mbb R^n$, and 
combine it with the result from the first regime. In the third regime,  we 
borrow a recent fractal Fourier restriction theorem of Du and Zhang and 
combine it with the result from the second regime. In the opposite 
direction, the results of this paper improve on the Du-Zhang theorem in the 
range $0 < \alpha < n/2$.
\end{abstract}

\maketitle

\section{Introduction}

Let $S$ be a smooth compact hypersurface in $\mbb R^n$ with a strictly 
positive second fundamental form, and $\sigma$ be the surface area measure 
on $S$. The extension operator $E=E_S$ on $S$ is defined as
\begin{displaymath}
Ef(x) = E_Sf(x) = \widehat{f d\sigma}(x)=
\int_S e^{-2 \pi i x \cdot \xi} f(\xi) d\sigma(\xi)
\end{displaymath}
for $f \in L^1(S)=L^1(\sigma)$. The restriction conjecture in harmonic 
analysis asserts that the operator $E$ is bounded from $L^p(S)$ to 
$L^q(\mbb R^n)$ whenever
\begin{equation}
\label{conjexp}
q > \frac{2n}{n-1} \;\;\;\;\; \mbox{ and } \;\;\;\;\; 
\frac{n-1}{p}+\frac{n+1}{q} \leq n-1.
\end{equation}
This conjecture is proved in the plane, but is largely open in higher
dimensions.

There are two important sets of exponents $(p,q)$ satisfying 
(\ref{conjexp}): $\{ (2,q) : q \geq (2n+2)/(n-1) \}$ and
$\{ (\infty,q) : q > 2n/(n-1) \}$. For the first set of exponents, the 
restriction conjecture is known to be true in all dimensions $n \geq 2$. In 
other words, the estimate
\begin{equation}
\label{tomastein}
\| E f \|_{L^q(\mbb R^n)} \lct \| f \|_{L^2(S)} 
\end{equation}
holds (uniformly in $f$) for $q \geq (2n+2)/(n-1)$. This result is known in 
the literature as the Tomas-Stein restriction theorem.

For the second set of exponents, there are only partial results. We know 
that the estimate
\begin{equation}
\label{p=infty}
\| E f \|_{L^q(\mbb R^n)} \lct \| f \|_{L^\infty(S)} 
\end{equation}
holds for $q > 13/4$ when $n=3$ (see \cite{guth:poly}), 
$q > (14/5)-(2/416515)$ when $n=4$ (see \cite{demeter:dim4} and
\cite{jz:severi}), $q > 2(3n+1)/(3n-3)$ when $n \geq 5$ is odd, and
$q > 2(3n+2)/(3n-2)$ when $n \geq 6$ is even (see \cite{guth:poly2} and
\cite{ghi:variable}). (For a recent improvement in $\mbb R^3$, see
\cite{w:brooms}; and in $\mbb R^n$, $n \geq 4$, see \cite{hr:highdim}.)

We also refer the reader to \cite{plms12046} and \cite{jk:someremarks} for
the full range of $p$ and $q$ exponents corresponding to Guth's $q > 13/4$
result in \cite{guth:poly}.

Suppose $n \geq 1$ and $0 < \alpha \leq n$. For Lebesgue measurable 
functions $H: \mbb R^n \to [0,1]$, we define
\begin{displaymath}
A_\alpha(H)= \inf \Big\{ C : \int_{B(x_0,R)} H(x) dx \leq C R^\alpha
\mbox{ for all } x_0 \in \mbb R^n \mbox{ and } R \geq 1 \Big\},
\end{displaymath}
where $B(x_0,R)$ denotes the closed ball in $\mbb R^n$ of center $x_0$ and
radius $R$. We say $H$ is a {\it weight of (fractal) dimension} $\alpha$ if 
$A_\alpha(H) < \infty$. We note that $A_\beta(H) \leq A_\alpha(H)$ if 
$\beta \geq \alpha$, so we are not really assigning a dimension to the 
function $H$; the phrase ``$H$ is a weight of dimension $\alpha$'' is 
merely another way for us to say that $A_\alpha(H) < \infty$. (The 
motivation for referring to $\alpha$ as a fractal dimension comes from 
Sections 4 and 8 below.)

We are interested in weighted restriction estimates of the form
\begin{equation}
\label{genericweighted}
\| Ef \|_{L^q(Hdx)} \lct A_\alpha(H)^{1/q} \| f \|_{L^p(S)}
\end{equation}
that hold uniformly in $f$ and $H$. In other words, the implicit constant in
(\ref{genericweighted}) is allowed to depend on the exponents $p$ and $q$, 
the dimensions $\alpha$ and $n$, and the surface $S$; but must be 
independent of the functions $f$ on $S$ and the weights $H$ on $\mbb R^n$.
We shall refer to (\ref{genericweighted}) as a weighted $L^p$-based 
estimate. 

One of the main goals of this paper is to prove the following theorem, and
then use it in $\mbb R^n$, $n \geq 2$, to obtain new results concerning 
weighted $L^2$-based and $L^\infty$-based restriction estimates.

For Lebesgue measurable functions $F : \mbb R^n \to [0, \infty)$, we define
\begin{displaymath}
M_\alpha F = 
\Big( \sup_H \frac{1}{A_\alpha(H)} \int F(x)^\alpha H(x)dx \Big)^{1/\alpha},
\end{displaymath}
where $H$ ranges over all non-zero weights on $\mbb R^n$ of dimension 
$\alpha$. We prove the following weighted H\"{o}lder-type inequality that
might be of independent interest.

\begin{thm}
\label{whtineq}
Suppose $n \geq 1$ and $0 < \beta < \alpha \leq n$. Then
\begin{equation}
\label{abstractineq}
M_\alpha F \leq M_\beta F
\end{equation}
for all non-negative Lebesgue measurable functions $F$ on $\mbb R^n$.
\end{thm}

Before we present the rest of our results, we give a couple of examples 
that are meant to provide the reader with a quick overview of the main 
theme of the paper concerning $L^2$-based estimates. In both examples, 
which take place in the plane, $S$ will be the unit circle.

Consider the set
$X= \{ (x,y) \in \mbb R^2 : x > 0 \mbox{ and } 0 \leq y \leq x^{-1/4} \}$.
It is easy to see that the characteristic function $\chi_X$ is a weight on 
$\mbb R^2$ of dimension $3/4$. We want to determine the best range of
exponents $q$ for which the following restriction estimate holds:
\begin{equation}
\label{themeone}
\| Ef \|_{L^q(X)} \lct \| f \|_{L^2(S)}.
\end{equation}

To every $R > 1$ there is a function $f_R$ on $S$ such that 
$\| f_R \|_{L^2(S)} \lct R^{-1/4}$ and $|Ef_R| \gct R^{-1/2}$ on the 
rectangle $[0,R] \times [0,\sqrt{R}]$. The intersection of this rectangle 
with $X$ contains the rectangle $[0,R] \times [0,R^{-1/4}]$, and hence
$\| Ef_R \|_{L^q(X)}$ $\gct$ $R^{(-1/2)+(3/(4q))}$. So the exponent $q$ 
in (\ref{themeone}) must satisfy $(-1/2)+(3/(4q)) \leq -1/4$, and so a
necessary condition for (\ref{themeone}) to hold is $q \geq 3$, which is far
from the sufficient condition $q \geq 6$ guaranteed by (\ref{tomastein}). 
Even the $L^\infty$-based estimate (\ref{p=infty}) only gives the 
sufficient condition $q > 4$ in the plane.

In the second example, we consider the set
$Y= \cup_{l=1}^\infty \mbb R \times [l^2,1+l^2]$, and we observe that the 
characteristic function $\chi_Y$ is a weight on $\mbb R^2$ of dimension 
$3/2$. Again, we want to determine the best range of exponents $q$ for 
which the following restriction estimate holds:
\begin{equation}
\label{themetwo}
\| Ef \|_{L^q(Y)} \lct \| f \|_{L^2(S)}.
\end{equation}

For $R > 1$, let $f_R$ be the same function on $S$ that was defined during 
the first example. Then $|Ef_R| \gct R^{-1/2}$ on every rectangle
$[0,R] \times [l^2,1+l^2]$ with $l^2 \leq \sqrt{R}$. Since there are 
$\sim R^{1/4}$ such rectangles, we see that 
$\| Ef_R \|_{L^q(Y)} \gct R^{(-1/2)+(5/(4q))}$. So the exponent $q$ in 
(\ref{themetwo}) must satisfy $(-1/2)+(5/(4q)) \leq -1/4$, and so a 
necessary condition for (\ref{themetwo}) is $q \geq 5$, which is again far 
from the sufficient condition $q \geq 6$ guaranteed by (\ref{tomastein}).

The results of this paper will show that (\ref{themeone}) and 
(\ref{themetwo}) indeed hold for $q > 3$ and $q > 5$, respectively. As it 
turns out, we can establish these sharp (up to the endpoints $q=3$ and
$q=5$) estimates on $X$ and $Y$ as follows.

We first prove a weighted restriction estimate
\begin{displaymath}
\| Ef \|_{L^{q_1}(Hdx)} \lct A_\beta(H)^{1/q_1} \| f \|_{L^2(S)}
\end{displaymath}
that holds whenever $0 < \beta < 1/2$ and $q_1 > 2$, and then combine it 
with the weighted H\"{o}lder-type inequality of Theorem \ref{whtineq} to
conclude that (\ref{themeone}) holds for $q > 3$. In doing so, we realize
that the same argument shows that the estimate
\begin{displaymath}
\| Ef \|_{L^{q_2}(Hdx)} \lct A_\alpha(H)^{1/q_2} \| f \|_{L^2(S)}
\end{displaymath}
holds whenever $1/2 \leq \alpha \leq 1$ and $q_2 > 4 \alpha$. Combining the 
last estimate with a corollary (see Corollary \ref{highfracdimcoro}) of the 
fractal restriction theorem of Du and Zhang \cite{dz:schrodinger3}, we see 
that (\ref{themetwo}) holds for $q > 5$. For more details, we refer the 
reader to Theorem \ref{main1} and Subsection 3.4.

The strategy that we explore in this paper of proving restriction estimates 
on specific sets, such as the sets $X$ and $Y$ in the above examples, by 
first proving restriction estimates for all weights $H$ of low fractal 
dimensions and then upgrading the estimates to higher fractal dimensions is 
reminiscent of the polynomial method of \cite{guth:poly} and 
\cite{guth:poly2} that upgrades restriction estimates from low algebraic 
dimensions to higher ones.  

When it comes to weighted $L^\infty$-based estimates, i.e.\
$L^\infty(S) \to L^q(Hdx)$ estimates, the situation becomes much harder, 
and we will postpone that discussion to the next section.

\section{Results and methodology}

Any restriction estimate $\| E f \|_{L^q(\mbb R^n)} \lct \| f \|_{L^p(S)}$
is equivalent to the weighted estimate
$\| E f \|_{L^q(Hdx)} \lct A_n(H)^{1/q} \| f \|_{L^p(S)}$. In fact, taking 
$H=1$, we see that the latter estimate implies the former. On the other 
hand, since the surface $S$ is compact, we can find a $C_0^\infty$ function 
$\phi$ on $\mbb R^n$ that satisfies $|\phi| \geq 1$ on $S$ and 
$\widehat{\phi}$ is compactly supported. Given $f \in L^p(S)$, we define
$g \in L^p(S)$ by $g=f/\phi$, and we observe that $|g| \leq |f|$,
$E f = (E g) \ast \widehat{\phi}$, and 
$|E f|^q \lct |E g|^q \ast |\widehat{\phi}|$. The non-weighted estimate 
applied to $g$ then tells us that 
\begin{eqnarray*}
\lefteqn{\int |E f(x)|^q H(x) dx
\lct \int |E g(y)|^q \int |\widehat{\phi}(x-y)| H(x) dx dy} \\ 
& & \lct A_n(H) \int |E g(y)|^q dy \lct A_n(H) \| g \|_{L^p(S)}^q 
    \leq A_n(H) \| f \|_{L^p(S)}^q.
\end{eqnarray*}
In particular, the Tomas-Stein estimate (\ref{tomastein}) has the following
weighted version:
\begin{equation}
\label{weightomastein}
\int |E f(x)|^q H(x) dx \lct A_\alpha(H) \| f \|_{L^2(S)}^q
\end{equation}
for $0 < \alpha \leq n$ and $q \geq (2n+2)/(n-1)$, where we have used the 
fact that $A_n(H) \leq A_\alpha(H)$.

\begin{remark}
\label{curvrelax}
For establishing {\rm (\ref{tomastein})} {\rm (}and hence {\rm 
(\ref{weightomastein})}{\rm )}, the assumption requiring the surface $S$ to 
have a strictly positive second fundamental can be relaxed to just 
requiring $S$ to have a nowhere vanishing Gaussian curvature. 
\end{remark}

With the restriction conjecture being open, it is therefore natural to
investigate the situation when $\alpha < n$. This has been the subject of 
two recent papers \cite{dgowwz:falconer} and \cite{plms12046}. Both papers 
employed Guth's polynomial partitioning method from \cite{guth:poly} and 
\cite{guth:poly2}.

We would like to mention at this point that weighted estimates of the form
\begin{displaymath}
\int_{B(0,R)} |E f(x)|^q H(x) dx \leq C A_\alpha(H) R^\beta
\| f \|_{L^2(S)}^q \hspace{0.5in} (R \geq 1),
\end{displaymath}
where we integrate over the ball $B(0,R)$ instead of $\mbb R^n$, and allow a
positive power of the radius $R$ on the right-hand side of the estimate, 
have been studied extensively in the literature due to their important 
applications in studying decay properties of Fourier transforms of measures 
and, consequently, Falconer's conjecture concerning distance sets in 
geometric measure theory. For such results, we refer the reader to 
\cite{pm:distsets}, \cite{tw:csoipaper}, \cite{tm:generalmeas},
\cite{mbe:notefractmeas}, \cite{mbe:birestfal}, \cite{rogerluca}, 
\cite{plms12046}, \cite{dgowwz:falconer}, \cite{dz:schrodinger3}, 
\cite{th:conical} and the references contained within these papers. The 
results of the present paper do not lead to any progress in the direction 
of Falconer's conjecture. In fact, and as we mentioned in the paper's 
abstract and introduction, our main concern is to study the restriction 
problem on subsets $X$ of $\mbb R^n$ with Lebesgue measure $|X|=\infty$ and 
satisfying the dimentionality property: $|X \cap B_R| \lct R^\alpha$ for 
all balls $B_R$ in $\mbb R^n$ of radius $R \geq 1$. We will, however, use 
known facts about the decay properties of the spherical means of Fourier 
transforms of measures to establish some of the lower bounds in Theorems 
\ref{main2} and \ref{main3}.

As the title of the present paper indicates, we are here mostly interested 
in studying the restriction problem in low fractal dimensions. In fact, 
this paper proves new weighted $L^2$-based restriction estimates, i.e.\ 
$L^2(S) \to L^q(Hdx)$ estimates, in $\mbb R^n$, $n \geq 3$, for 
$0 < \alpha \leq (n+1)/2$ (see Theorem \ref{main1}). In particular, if $X$
is as in the previous paragraph, then, taking $H$ to be the characteristic
function of $X$, we get new $L^2(S) \to L^q(X)$ restriction estimates.

In the plane, we prove new weighted $L^2$-based restriction estimates in 
the full range $0 < \alpha < 2$ of fractal dimensions. This is one 
important aspect of the approach we follow, because the results of 
\cite{dgowwz:falconer} and \cite{plms12046} do not include the plane. 

In the regime $0 < \alpha \leq n/2$, the best known weighted $L^2$-based 
restriction estimates were obtained in \cite{plms12046} for $n=3$, and in 
\cite{dgowwz:falconer} for $n \geq 3$. 

The authors of \cite{dgowwz:falconer} proved that in $\mbb R^n$, 
$n \geq 3$, to every $\epsilon > 0$ there is a constant $C_\epsilon$ such 
that
\begin{equation}
\label{knownest}
\int_{B(0,R)} |E f(x)|^{2n/(n-1)} H(x) dx \leq C_\epsilon R^\epsilon 
A_\alpha(H) \| f \|_{L^2(S)}^{2n/(n-1)}
\end{equation}
whenever $f \in L^2(S)$, $0 < \alpha \leq \max[n/2,2]$, $H$ is a weight of 
dimension $\alpha$, and $R \geq 1$. 
(See \cite[Theorem 1.8 and Remark 1.10]{dgowwz:falconer}.) In 
(\ref{knownest}), the constant $C_\epsilon$ is only allowed to depend on 
$\epsilon$, $\alpha$, $n$, and $S$. Estimates such as (\ref{knownest}), 
where one integrates $E f$ over a ball of radius $R$ instead of the entire 
$\mbb R^n$, are often referred to in the literature as {\it local 
restriction estimates}. Also, to emphasize the fact that the function $E f$ 
is being integrated over $\mbb R^n$, estimates such as (\ref{p=infty}) and 
(\ref{weightomastein}) are often called {\it global restriction estimates}.

In \cite{plms12046}, (\ref{knownest}) was proved in $\mbb R^3$, but only 
for $0 < \alpha \leq 3/2$ and with $A_\alpha(H)$ replaced by 
$\max[A_\alpha(H),A_\alpha(H)^{1/4}]$.

\begin{remark}
In \cite{dgowwz:falconer}, weights were defined in a slightly different way
than in this paper. For $0 < \alpha \leq n$, the authors of 
\cite{dgowwz:falconer} denoted by ${\mathcal F}_{\alpha,n}$ the set of all
non-negative measurable functions $H$ on $\mbb R^n$ that satisfy
$\int_{B(x_0,R)} H(x) dx \leq R^\alpha$ for all $x_0 \in \mbb R^n$ and 
$R \geq 1$, and wrote {\rm (\ref{knownest})} as:
\begin{equation}
\label{knownestprime}
\int_{B(0,R)} |E f(x)|^{2n/(n-1)} H(x) dx \leq C_\epsilon 
\| f \|_{L^2(S)}^{2n/(n-1)}
\end{equation}
whenever $f \in L^2(S)$, $0 < \alpha \leq n/2$, 
$H \in {\mathcal F}_{\alpha,n}$, and $R \geq 1$. The estimates {\rm 
(\ref{knownest})} and {\rm (\ref{knownestprime})} are equivalent. Clearly, 
{\rm (\ref{knownestprime})} implies {\rm (\ref{knownest})}. To establish 
the reverse implication, given $H \in {\mathcal F}_{\alpha,n}$ and 
$N \in \mbb N$, we let 
$H_N=N^{-1} \chi_{\{ x \in \mbb R^n: H(x) \leq N \}} H$, observe that
$A_\alpha(H_N) \leq N^{-1}$, apply {\rm (\ref{knownest})} with $H_N$, send 
$N$ to infinity, and arrive at {\rm (\ref{knownestprime})} via the monotone 
convergence theorem.
\end{remark}

The polynomial method for proving restriction estimates that was developed 
in \cite{guth:poly} and \cite{guth:poly2} in the non-weighted setting, and 
adapted in \cite{dgowwz:falconer} and \cite{plms12046} to the weighted 
setting, cannot prove restriction estimates for exponents $q < 2n/(n-1)$. 
In fact, the polynomial method has a key induction argument in the 
non-algebraic (or cellular) case in which the condition $q \geq 2n/(n-1)$ 
is crucial for closing the induction. Since one naturally expects $q$ to go 
below $2n/(n-1)$ as $\alpha$ becomes smaller (and, as one learns from 
Theorem \ref{main1}, turns out to indeed be the case), the polynomial 
method does not appear to be of much help in handling the 
$0 < \alpha < n/2$ case.

Also, the polynomial method proves local restriction estimates. In the
non-weighted setting, this is not a serious limitation, because one can turn
local restriction estimates into global ones by using Tao's 
$\epsilon$-removal lemma from \cite{tt:removal}. In the weighted setting, 
however, the $\epsilon$-removal lemma can only be applied in some special 
cases (see \cite[Section 2]{plms12046}).

In this paper, we prove global weighted $L^2$-based restriction estimates,
and we manage to go below the $2n/(n-1)$ threshold (when $0< \alpha < n/2$) 
as follows. We divide $0 < \alpha \leq n$ into three regimes:
$0 < \alpha < (n-1)/2$, $(n-1)/2 \leq \alpha \leq n/2$, and 
$n/2 < \alpha \leq n$. 

In the first regime, we prove an $L^2(S) \to L^q(Hdx)$ restriction estimate 
that holds for all $q > 2$, and which is sharp up to the endpoint $q=2$. In 
this part of the proof, we use ideas from Bourgain's paper 
\cite{jb:besicovitch} to utilize the decay we have on the Fourier 
transform of the surface measure $\sigma$ on $S$.

In the second regime, we prove an $L^2(S) \to L^q(Hdx)$ restriction 
estimate that holds for all $q > 4\alpha/(n-1)$. To obtain this result, we 
combine the result that we have obtained in the first regime with a 
corollary of Theorem \ref{whtineq} (see Corollary \ref{indconclusion} 
below).

Once we have established our restriction estimates in the regime 
$(n-1)/2 \leq \alpha \leq n/2$ via Corollary \ref{indconclusion}, we combine
them with the fractal restriction theorem of Du and Zhang 
\cite{dz:schrodinger3} to obtain new $L^2(S) \to L^q(Hdx)$ estimates in the 
regime $n/2 < \alpha \leq (n+1)/2$ for $n \geq 3$, and $n/2 < \alpha < n$ 
for $n=2$. As will become apparent during the proof of Theorem \ref{main1}, 
in the plane we will able to use the full strength of the theorem of Du and 
Zhang, but in $\mbb R^n$, $n \geq 3$, we will need to weaken the Du-Zhang 
theorem before we can combine it with the estimates from the second regime. 

In the opposite direction, it turns out that our $L^2$-based estimates 
actually improve on the fractal restriction theorem of 
\cite{dz:schrodinger3} when $0 < \alpha < n/2$ (see Corollary
\ref{coromain1}).
 
Here are the main results of this paper concerning $L^2$-based estimates.

\begin{thm}
\label{main1}
Suppose $n \geq 2$ and $S$ is a smooth compact hypersurface in $\mbb R^n$ 
with a strictly positive second fundamental form. Then
\begin{displaymath}
\int |E f(x)|^q H(x) dx \lct A_\alpha(H) \| f \|_{L^2(S)}^q
\end{displaymath}
for all functions $f \in L^2(S)$ and weights $H$ on $\mbb R^n$ of dimension 
$\alpha$ whenever
\begin{displaymath}
q > \left\{ \begin{array}{ll}
2 & \mbox{ if \, $0 < \alpha < (n-1)/2$,} \\
4\alpha/(n-1) & \mbox{ if \, $(n-1)/2 \leq \alpha \leq n/2$,} \\
2 \alpha + 2 & \mbox{ if \, $n = 2$ and $1 < \alpha < 2$,} \\
(2n/(n-1)) + 2 - (n/\alpha) & \mbox{ if \, $n \geq 3$ and 
                                     $n/2 < \alpha \leq (n+1)/2$.}
\end{array} \right.
\end{displaymath}
\end{thm}

We remark to the reader that the $q > (2n/(n-1))+ 2 - (n/\alpha)$ result of 
Theorem \ref{main1} is in fact true for $(n-1)/2 \leq \alpha \leq n$, but 
in the regime $(n+1)/2 < \alpha < n$ it becomes inferior to the estimate in 
Proposition \ref{simbase} that we state and prove in Section 6 below, and at
$\alpha=n$ it becomes inferior to the Tomas-Stein estimate 
(\ref{weightomastein}).

The reason for not stating Proposition \ref{simbase} here is due to the fact
its proof does not follow the strategy outlined above. Instead, the proof of
Proposition \ref{simbase} combines the Du-Zhang estimate from 
\cite{dz:schrodinger3} with the method that Bourgain developed in 
\cite{jb:besicovitch} to upgrade local restriction estimates to global ones.

The assumption that the surface $S$ has a strictly positive second 
fundamental form is only needed for the $q > 2\alpha +2$ and
$q > (2n/(n-1))+ 2 - (n/\alpha)$ results of Theorem \ref{main1} (because of 
the need to use the fractal restriction theorem of \cite{dz:schrodinger3}). 
For the other two results, we only need $S$ to have a nowhere vanishing 
Gaussian curvature, as will become clear during the proof of Theorem 
\ref{main1} (see also Remark \ref{curvrelax} and Proposition \ref{base}). 

The ranges of the exponent $q$ in Theorem \ref{main1} are all sharp (up to 
the endpoints) in $\mbb R^2$. In $\mbb R^n$, $n \geq 3$, we are only able 
to show that the $q > 2$ range is sharp (again up to the endpoint). These 
results are detailed in the following theorem.

\begin{thm}
\label{main2}
Let $S$ be the unit sphere in $\mbb R^n$. Suppose that to every 
$\epsilon>0$ there is a constant $C_\epsilon$ such that
\begin{displaymath}
\int_{B(0,R)} |Ef(x)|^q H(x) dx \leq C_\epsilon R^\epsilon A_\alpha(H)
\| f \|_{L^2(S)}^q	
\end{displaymath}
for all functions $f \in L^2(S)$, weights $H$ on $\mbb R^n$ of dimension 
$\alpha$, and radii $R \geq 1$. Then
\begin{displaymath}
q \geq \left\{ \begin{array}{ll}
2 & \mbox{ if $n \geq 3$ and $0 < \alpha \leq n-2$,} \\
(2\alpha+2)/(n-1) & \mbox{ if $n \geq 2$ and $n-2 < \alpha \leq n$,} \\
2 & \mbox{ if $n = 2$ and $0 < \alpha < 1/2$,} \\
4 \alpha & \mbox{ if $n = 2$ and $1/2 \leq \alpha \leq 1$.} 
\end{array} \right.
\end{displaymath} 
\end{thm}

Before starting the discussion of weighted $L^\infty$-based estimates, we 
make a couple of definitions and state a corollary of Theorem \ref{whtineq}.

For $0 < \alpha \leq n$ and $1 \leq p \leq \infty$, we define $Q(\alpha,p)$ 
to be the infimum of all numbers $q > 0$ such that the following holds:
there is a constant $C$ such that
\begin{displaymath}
\int |Ef(x)|^q H(x) dx \leq C A_\alpha(H) \| f \|_{L^p(S)}^q
\end{displaymath}
for all functions $f \in L^p(S)$ and weights $H$ on $\mbb R^n$ of dimension 
$\alpha$. The constant $C$ is allowed to depend on $n$, $\alpha$, $p$, and 
$q$; but, of course, not on $f$ or $H$. 

We also define $Q_{\rm loc}(\alpha,p)$ to be the infimum of all numbers 
$q > 0$ such that the following holds: to every $\epsilon > 0$ there is a 
constant $C_\epsilon$ such that
\begin{displaymath}
\int_{B(0,R)} |Ef(x)|^q H(x) dx \leq C_\epsilon R^\epsilon A_\alpha(H) 
\| f \|_{L^p(S)}^q
\end{displaymath}
for all functions $f \in L^p(S)$, weights $H$ on $\mbb R^n$ of dimension 
$\alpha$, and radii $R \geq 1$. The constant $C_\epsilon$ is allowed to 
depend on $\epsilon$, $n$, $\alpha$, $p$, and $q$.

For applications in the Fourier restriction context, it will be convenient
to state the following corollary of Theorem \ref{whtineq}.

\begin{coro}
\label{indconclusion}
Suppose $n \geq 2$ and $0 < \beta < \alpha \leq n$. Then 
\begin{displaymath}
\frac{Q(\alpha,p)}{\alpha} \leq \frac{Q(\beta,p)}{\beta}
\hspace{0.25in} \mbox{and} \hspace{0.25in}
\frac{Q_{\rm loc}(\alpha,p)}{\alpha} 
\leq \frac{Q_{\rm loc}(\beta,p)}{\beta}.
\end{displaymath}
\end{coro}

In view of the fact that Corollary \ref{indconclusion} holds for all
$1 \leq p \leq \infty$, the strategy we outlined above for deriving 
restriction estimates by breaking $0< \alpha \leq n$ into different regimes 
works as well for $L^\infty$-based estimates as it did for $L^2$-based 
estimates. But, unlike the $L^2$-based situation, we are unable to prove a 
favorable $L^\infty$-based estimate for small $\alpha$. In fact, 
establishing a local $L^\infty(S) \to L^1(\chi_{B(0,R)}Hdx)$ restriction
estimate for $0 < \alpha < (n-1)/2$ would imply (via Corollary 
\ref{indconclusion}) a local $L^\infty(S) \to L^{2n/(n-1)}(B(0,R))$ 
estimate, which would essentially solve the restriction problem in 
$\mbb R^n$. So it becomes natural to investigate if such an estimate is 
feasible. For example, could $Q_{\rm loc}(\alpha,\infty) \to 0$ as 
$\alpha \to 0$? The next theorem tells us that 
$Q_{\rm loc}(\alpha,\infty) \geq (n-1)/n$ for small $\alpha$, but proving 
this lower bound turned out to be much harder than the author had initially 
expected. 

\begin{thm}
\label{main3}
Let $S$ be the unit sphere in $\mbb R^n$. Then
\begin{displaymath}
Q_{\rm loc}(\alpha,\infty) \geq \left\{ \begin{array}{ll}
(n-1)/n & \mbox{ if $n \geq 3$ and $0 < \alpha \leq (n-1)^2/(2n)$,} \\
2\alpha/(n-1) & \mbox{ if $n \geq 3$ and 
$(n-1)^2/(2n) \leq \alpha \leq n$,} 
\\
1/2 & \mbox{ if $n = 2$ and $0 < \alpha \leq 1/6$,} \\
3\alpha & \mbox{ if $n = 2$ and $1/6 \leq \alpha \leq 1$,} \\
\alpha + 2 & \mbox{ if $n=2$ and $1 \leq \alpha \leq 2$.} 
\end{array} \right.
\end{displaymath} 
\end{thm}

The proofs of the first and last inequalities in Theorem \ref{main3} (as 
well as the first inequality in Theorem \ref{main2}) involve some geometric 
measure theory. In particular, the last inequality depends on a theorem of 
Bennett and Vargas \cite{bv:randomised} about the decay of the $L^1$ 
circular means of Fourier transforms of measures.

The rest of the paper is organized as follows. In the next section, we give 
some interesting examples. In Section 4, we discuss the fractal restriction
theorem of Du and Zhang \cite{dz:schrodinger3} and show how Theorem 
\ref{main1} improves on it when $0 < \alpha < n/2$. Section 5 is dedicated 
to the proofs of Theorem \ref{whtineq} and Corollary \ref{indconclusion}. 
Section 6 proves estimates in the regimes $0 < \alpha < (n-1)/2$ and
$(n+1)/2 < \alpha < n$. The last four sections of the paper are the proofs 
of Theorems \ref{main1}, \ref{main2}, and \ref{main3}.

\section{Examples}

\subsection{Restriction estimates in neighborhoods of algebraic varieties}

Let $n \geq 2$, $Z$ be a real algebraic variety in $\mbb R^n$ of dimension 
$k$ that is defined by polynomials of degree at most $D$, and $N_\rho(Z)$ 
be the $\rho$-neighborhood of $Z$. Then a theorem of Wongkew 
\cite{w:tubular} tells us that
\begin{displaymath}
|N_\rho(Z) \cap B_R| \leq C_n (D \rho)^{n-k} R^k
\end{displaymath}
for any ball $B_R \subset \mbb R^n$ of radius $R > 0$, where $C_n$ is a 
constant that depends only on $n$. This inequality implies that the 
characteristic function of $N_\rho(Z)$ is a weight on $\mbb R^n$ of 
dimension $k$. Moreover, $A_k(\chi_{N_\rho(Z)}) \leq C_n (D \rho)^{n-k}$. 
Therefore, if $n \geq 3$ and $1 \leq k \leq (n+1)/2$, then Theorem 
\ref{main1} applies and tells us that
\begin{displaymath}
\int_{N_\rho(Z)} |E f(x)|^q dx \lct (D \rho)^{n-k} \| f \|_{L^2(S)}^q
\end{displaymath}
for all $f \in L^2(S)$ whenever
\begin{displaymath}
q > \left\{ \begin{array}{ll}
\max[2, 4k/(n-1)] & \mbox{ if \, $1 \leq k \leq n/2$,} \\
(2n/(n-1))+2-(n/k) & \mbox{ if \, $n/2 < k \leq (n+1)/2$.} 
\end{array} \right.
\end{displaymath}
Furthermore, if $n=2$ and $Z$ is the zero set of a polynomial $P$ in two 
real variables of degree $D \geq 1$, then Theorem \ref{main1} gives the 
estimate
\begin{displaymath}
\int_{N_\rho(Z)} |E f(x)|^q dx \lct D \rho \, \| f \|_{L^2(S)}^q
\hspace{0.5in} (q > 4).
\end{displaymath}
We also refer the reader to the example at the end of Section 6 for a 
similar estimate in higher dimensions.

\subsection{An example of the $\alpha=n/2$ case}

As a second example, we consider the set $\Omega \subset \mbb R^n$ given by
\begin{displaymath}
\Omega= \{ x=(\bar{x},x_n) \in \mbb R^{n-1} \times \mbb R :
|x_n| \leq |\bar{x}|^{1-(n/2)} \}.
\end{displaymath}
It is easy to see that the characteristic function of $\Omega$ is a weight 
on $\mbb R^n$ of dimension $n/2$ and with $A_{n/2}(\chi_\Omega) \lct 1$. In 
fact, if $\bar{x}_0 \in \mbb R^{n-1}$ and
$R \geq 1$, then
\begin{displaymath}
\int_{B(\bar{x}_0,R)} |\bar{x}|^{1-(n/2)} d\bar{x} \leq \left\{
\begin{array}{ll}
\int_{B(0,11R)} |\bar{x}|^{1-(n/2)} d\bar{x} \lct R^{n/2}	& 
\mbox{ if $|\bar{x}_0| \leq 10R$,} \\ \\
(|\bar{x}_0|-R)^{1-(n/2)} \int_{B(\bar{x}_0,R)} dx \lct R^{n/2} & 
\mbox{ if $|\bar{x}_0| > 10R$.}
\end{array} \right.
\end{displaymath}
Therefore, (\ref{knownest}) gives us the local estimate
\begin{equation}
\label{localomega}
\int_{\Omega \cap B_R} |E f(x)|^{2n/(n-1)} dx
\lct R^\epsilon \| f \|_{L^2(S)}^{2n/(n-1)},
\end{equation}
whereas Theorem \ref{main1} gives us the global estimate
\begin{equation}
\label{globalomega}
\int_\Omega |E f(x)|^q dx \lct \| f \|_{L^2(S)}^q
\hspace{0.5in} (q > 2n/(n-1)).
\end{equation}
(See Remark \ref{weightedremoval} following the next example.)

\subsection{An example in $\mbb R^3$}

Our third example, which takes place in $\mbb R^3$, needs the following 
result from \cite{dgowwz:falconer}:
\begin{displaymath}
\int_{B_R} |E f(x)|^3 H(x) dx \leq C_\epsilon R^\epsilon 
A_\alpha(H) \| f \|_{L^2(S)}^3
\end{displaymath}
for all functions $f \in L^2(S)$ and weights $H$ on $\mbb R^3$ of dimension 
$3/2 \leq \alpha \leq 2$. This local estimate implies that 
$Q_{\rm loc}(2,2) \leq 3$, and so Corollary \ref{indconclusion} (applied 
with $\beta=2$ and $n=3$) implies that
\begin{equation}
\label{knownestr3}
Q_{\rm loc}(\alpha,2) \leq (3/2) \alpha \hspace{0.5in} 
(2 < \alpha \leq (11-\sqrt{13})/3).
\end{equation}
(Proposition \ref{simbase} gives a better result when
$(11-\sqrt{13})/3 < \alpha < 3$, and  so does (\ref{weightomastein}) 
when $\alpha=3$.) 
 
Now, for $0 < b \leq 1$, we define
\begin{displaymath}
\Omega_b = \cup_{l=1}^\infty \mbb R^2 \times [l^{1/b},1+l^{1/b}]
\end{displaymath}
and we observe that the characteristic function of $\Omega_b$ is a weight 
on $\mbb R^3$ of dimension $2+b$ and with 
$A_{2+b}(\chi_{\Omega_b}) \lct 1$. So (\ref{knownestr3}) tells us that
\begin{displaymath}
\int_{\Omega_b \cap B_R} |E f(x)|^q dx \lct R^\epsilon \| f \|_{L^2(S)}^q
\end{displaymath}
for all $f \in L^2(S)$ whenever $b \leq (5-\sqrt{13})/3$ and 
$q \geq (3/2)(2+b)$. Combining this local estimate with the 
$\epsilon$-removal argument of \cite[Corollary 2.1]{plms12046} (which is a 
small modification on Tao's $\epsilon$-removal lemma from \cite{tt:removal} 
that works for some special weights such as the characteristic function of 
$\Omega_b$), we conclude that the global estimate
\begin{displaymath}
\int_{\Omega_b} |E f(x)|^q dx \lct \| f \|_{L^2(S)}^q
\end{displaymath}
holds whenever $b \leq (5-\sqrt{13})/3$ and $q > (3/2)(2+b)$.

\begin{remark}
\label{weightedremoval}
The $\epsilon$-removal argument of \cite[Corollary 2.1]{plms12046} does not 
apply when $H$ is the characteristic function of the set $\Omega$ in the 
second example (Subsection 3.2), so the author is not sure whether {\rm 
(\ref{globalomega})} can be derived from {\rm (\ref{localomega})} when 
$n \geq 3$.
\end{remark}

\subsection{The two examples from the Introduction -- revisited}

For $0 \leq b \leq 1$, we define
\begin{displaymath}
X_b = \{ (x,y) \in \mbb R^2 : x > 0 \mbox{ and } 0 \leq y \leq x^{-b} \}.
\end{displaymath}
If $b < 1$, then the characteristic function of $X_b$ is a weight on 
$\mbb R^2$ of dimension $1-b$, and $A_{1-b}(\chi_{X_b}) \lct 1$. If $b=1$, 
then $\chi_{X_b}$ is a weight on $\mbb R^2$ of dimension $\alpha$ and 
$A_\alpha(\chi_{X_b}) \lct 1$ for all $0 < \alpha \leq 2$. Therefore, by 
Theorem \ref{main1}, we have
\begin{equation}
\label{xbest}
\int_{X_b} |E f(x)|^q dx \lct \| f \|_{L^2(S)}^q
\hspace{0.5in} (q > \max[2, 4(1-b)]).
\end{equation}
When $b=1$, we also have
\begin{displaymath}
\int_{X_1 \cap B_R} |E f(x)|^q dx \lct (\log R) \| f \|_{L^2(S)}^q
\hspace{0.5in} (q > 0, R \geq 1),
\end{displaymath}
but we do not have an $\epsilon$-removal theorem that would turn this local 
estimate into a global one.

We saw in the Introduction that the range of $q$ in (\ref{xbest}) is sharp 
(up to the end point $q=4(1-b)$) when $b=1/4$. We will see during the proof 
of Theorem \ref{main2} that this range of $q$ is actually sharp for all
$0 \leq b \leq 1/2$.

Our last example also takes place in the plane. For $0 < b < 1$, we define
\begin{displaymath}
Y_b = \cup_{l=1}^\infty \mbb R \times [l^{1/b},1+l^{1/b}].
\end{displaymath}
Then the characteristic function of $Y_b$ is a weight on $\mbb R^2$ of 
dimension $1+b$ and with $A_{1+b}(\chi_{Y_b}) \lct 1$. So Theorem 
\ref{main1} gives the estimate
\begin{equation}
\label{ybest}
\int_{Y_b} |E f(x)|^q dx \lct \| f \|_{L^2(S)}^q
\hspace{0.5in} (q > 4+2b).
\end{equation}

We saw in the Introduction that the range of $q$ in (\ref{ybest}) is sharp 
(up to the endpoint $q=4+2b$) when $b=1/2$. We will see during the proof of 
Theorem \ref{main2} that this range of $q$ is actually sharp for all
$0 < b < 1$. 
 
\section{On a fractal restriction theorem of Du and Zhang}

Throughout this section, we denote a cube in $\mbb R^n$ of center $x$ and
side-length $r$ by $\widetilde{B}(x,r)$. 

Let ${\mathcal P} = \{ \xi \in \mbb R^n : \xi_n = \xi_1^2 + \ldots + 
\xi_{n-1}^2 \leq 1 \}$ be the unit paraboloid in $\mbb R^n$, and 
$E_{\mathcal P}$ the extension operator associated with $\mathcal P$. In a 
recent paper \cite{dz:schrodinger3}, the following interesting theorem was 
proved.

\begin{alphthm}[{\cite[Corollary 1.6]{dz:schrodinger3}}]
\label{highfracdim}
Suppose $n \geq 2$, $1 \leq \alpha \leq n$, $R \geq 1$, 
$X= \cup_k \widetilde{B}_k$ is a union of lattice unit cubes in 
$\widetilde{B}(0,R) \subset \mbb R^n$, and
\begin{displaymath}
\gamma = \sup \frac{\#\{ \widetilde{B}_k : \widetilde{B}_k \subset 
\widetilde{B}(x',r) \}}{r^\alpha},
\end{displaymath}
where the sup is taken over all pairs 
$(x',r) \in \mbb R^n \times [1,\infty)$ satisfying 
$\widetilde{B}(x',r) \subset \widetilde{B}(0,R)$. Then to every 
$\epsilon > 0$ there is a constant $C_\epsilon$ such that
\begin{displaymath}
\| E_{\mathcal P} f \|_{L^2(X)} \leq C_\epsilon R^\epsilon \gamma^{1/n}
R^{\alpha/(2n)} \| f \|_{L^2({\mathcal P})}
\end{displaymath}
for all $f \in L^2({\mathcal P})$.
\end{alphthm}

Theorem \ref{highfracdim} is of interest to us in two ways. 

First, Theorem \ref{main1} of this paper allows us to improve the exponent 
of $R$ in Theorem \ref{highfracdim} from $\alpha/(2n)$ to $0$ for
$0 < \alpha \leq (n-1)/2$, and to $(\alpha/2)-((n-1)/4)$ for 
$(n-1)/2 < \alpha < n/2$. (We note to the reader that Theorem 
\ref{highfracdim} is a corollary of the main theorem in 
\cite{dz:schrodinger3} (see \cite[Theorem 1.3]{dz:schrodinger3}); our 
results do not appear to improve on the main theorem.)

Second, Theorem \ref{highfracdim} has a corollary that will help us to 
prove Theorem \ref{main1} in the regime $n/2 < \alpha \leq (n+1)/2$, as 
well as Proposition \ref{simbase}. 

We start by proving the corollary to Theorem \ref{main1} that improves on
Theorem \ref{highfracdim} when $0 < \alpha < n/2$.

\begin{coro}
\label{coromain1}
Suppose $n$, $R$, $X$, and $\gamma$ are as in Theorem \ref{highfracdim}.
Also, suppose that $S$ is a smooth compact hypersurface in $\mbb R^n$ with a 
nowhere vanishing Gaussian curvature, and $E=E_S$ is the extension operator
on $S$. Then to every $\epsilon > 0$ there is a constant $C_\epsilon$ such 
that
\begin{displaymath}
\| E f \|_{L^2(X)} \leq C_\epsilon R^\epsilon \gamma^{1/2} R^{e(\alpha)}
\| f \|_{L^2({\mathcal S})}
\end{displaymath}
for all $f \in L^2({\mathcal S})$, where
\begin{displaymath}
e(\alpha)= \left\{ \begin{array}{ll}
0 & \mbox{ if $0 < \alpha \leq (n-1)/2$,} \\
(\alpha/2)-((n-1)/4) & \mbox{ if $(n-1)/2 < \alpha \leq n/2$.}
\end{array} \right.
\end{displaymath}
\end{coro}

\begin{proof}
Theorem \ref{main1} provides us with the local estimate
\begin{displaymath}
\int_{B(0,R)} |Ef(x)|^{q(\alpha)} H(x) dx \leq C_\epsilon R^\epsilon 
A_\alpha(H) \| f \|_{L^2(S)}^{q(\alpha)}
\end{displaymath}
with
\begin{displaymath}
q(\alpha)= \left\{ \begin{array}{ll}
2 & \mbox{ if $0 < \alpha \leq (n-1)/2$,} \\
4\alpha/(n-1) & \mbox{ if $(n-1)/2 < \alpha \leq n/2$.}
\end{array} \right.
\end{displaymath}

We let $H$ be the characteristic function of $X$. By the definition of
$\gamma$, we have
\begin{displaymath}
\int_{\widetilde{B}(x_0,r)} H(x) dx \leq \gamma \, (r+2)^\alpha
                                    \leq \gamma \, (3r)^\alpha
\end{displaymath}
for all $x_0 \in \mbb R^n$ and $r \geq 1$. Thus $H$ is a weight on 
$\mbb R^n$ of dimension $\alpha$, and $A_\alpha(H) \lct \gamma$. This
immediately proves the corollary for $0 < \alpha \leq (n-1)/2$.

For $(n-1)/2 < \alpha \leq n/2$, we apply H\"{o}lder's inequality, to get
\begin{eqnarray*}
\lefteqn{\int_X |Ef(x)|^2 dx \; = \; 
\int_{B(0,R)} |Ef(x)|^2 H(x) dx} \\
& \leq & \Big( C_\epsilon R^\epsilon A_\alpha(H) 
         \| f \|_{L^2(S)}^{q(\alpha)} \Big)^{2/q(\alpha)}
         \Big( \int_{B(0,R)} H(x) dx \Big)^{1-(2/q(\alpha))} \\
& \leq & \Big( C_\epsilon R^\epsilon A_\alpha(H) 
         \| f \|_{L^2(S)}^{q(\alpha)} \Big)^{2/q(\alpha)}
         \Big( A_\alpha(H) R^\alpha \Big)^{1-(2/q(\alpha))} \\
&  =   & \big( C_\epsilon R^\epsilon \big)^{2/q(\alpha)} A_\alpha(H) 
         R^{2 e(\alpha)} \| f \|_{L^2(S)}^2,
\end{eqnarray*}
as claimed.
\end{proof}

Theorem \ref{highfracdim} has the following corollary that will be needed to
prove Theorem \ref{main1} in the regime $n/2 < \alpha \leq (n+1)/2$, and 
Proposition \ref{simbase}.

\begin{alphcoro}[\cite{dz:schrodinger3}]
\label{highfracdimcoro}
Suppose $n \geq 2$, $1 \leq \alpha \leq n$, $S$ is a smooth compact 
hypersurface in $\mbb R^n$ with a strictly positive second fundamental form,
and $E=E_S$ is the extension operator on $S$. Then to every $\epsilon > 0$ 
there is a constant $C_\epsilon$ such that
\begin{displaymath}
\int_{B(0,R)} |E f(x)|^2 H(x) dx \leq C_\epsilon R^\epsilon 
A_\alpha(H) R^{\alpha/n} \| f \|_{L^2(S)}^2
\end{displaymath}
for all functions $f \in L^2(S)$, weights $H$ on $\mbb R^n$ of dimension 
$\alpha$, and radii $R \geq 1$.
\end{alphcoro}

Corollary \ref{highfracdimcoro} is not stated as such in 
\cite{dz:schrodinger3}, but is very similar to 
\cite[Theorem 2.3]{dz:schrodinger3}.
Also, the proof of \cite[Theorem 2.3]{dz:schrodinger3} is not explicitly 
given in that paper, because it is very similar to the proof of 
\cite[Theorem 2.2]{dz:schrodinger3}. Therefore, we will the present the 
proof of Corollary \ref{highfracdimcoro} here for the reader's convenience.

\begin{proof}[Proof of Corollary \ref{highfracdimcoro}
{\rm (\cite[Proof of Theorem 2.2]{dz:schrodinger3})}]
We may assume that $S$ is the paraboloid ${\mathcal P}$ that was defined at
the beginning of this section (see \cite[Remark 1.10]{dgowwz:falconer} and
\cite[Part (III) of Subsection 2.2]{dz:schrodinger3} for the justification
of this assumption).

We consider a covering $\{ \widetilde{B} \}$ of $B(0,R)$ by unit lattice 
cubes, and for each such cube we define
$v(\widetilde{B})=A_\alpha(H)^{-1} \int_{\widetilde{B}} H(x) dx$. Also, for 
$k=0,-1,-2, \ldots$, we set $V_k = \{ \widetilde{B} : 2^{k-1} < 
n^{-\alpha/2} v(\widetilde{B}) \leq 2^k \}$. Since each cube 
$\widetilde{B}$ is contained in a ball of radius $\sqrt{n}$, we have 
$v(\widetilde{B}) \leq n^{\alpha/2}$, so that
$\cup_k V_k \supset \cup \, \widetilde{B} \supset B(0,R)$.

Let $k_1$ be the sup of the set $\{ k \in \mbb Z :2^k \leq R^{-1000n} \}$.
By the pigeonhole principle, there is an integer $k_0$ satisfying
$k_1 < k_0 \leq 0$ such that 
\begin{eqnarray}
\label{pigeonholing}
\lefteqn{\int_{B(0,R)} |Ef(x)|^2 H(x) dx}  \nonumber \\
& \lct & R^\epsilon 2^{k_0} A_\alpha(H) \sum_{\widetilde{B} \in V_{k_0}} 
\sup_{\widetilde{B}} |Ef|^2 + A_\alpha(H) \| f \|_{L^2(S)}^2 O(R^{-500n}).
\end{eqnarray}

Since the measure $fd\sigma$ is compactly supported and 
$Ef=\widehat{fd\sigma}$, there is a non-negative rapidly decaying function
$\psi$ on $\mbb R^n$ such that
$\sup_{\widetilde{B}} |Ef|^2 \leq |Ef|^2 \ast \psi(c(\widetilde{B}))$,
where $c(\widetilde{B})$ is the center of $\widetilde{B}$. So
\begin{displaymath}
\sup_{\widetilde{B}} |Ef|^2 \lct \int_{B(c(\widetilde{B}),R^\epsilon)} 
|Ef(x)|^2 dx + \| f \|_{L^2(S)}^2 O(R^{-1000n}),
\end{displaymath}
and since
\begin{equation}
\label{bdoverlap}
\sum_{\widetilde{B} \in V_{k_0}} \chi_{B(c(\widetilde{B}),R^\epsilon)} 
\lct R^{n\epsilon},
\end{equation}
it follows that
\begin{equation}
\label{localstar}
\sum_{\widetilde{B} \in V_{k_0}} \sup_{\widetilde{B}} |Ef|^2
\lct R^{n\epsilon} \int_V |Ef(x)|^2 dx + \| f \|_{L^2(S)}^2 O(R^{-500n}),
\end{equation}
where $V= \cup_{\widetilde{B} \in V_{k_0}} B(c(\widetilde{B}),R^\epsilon)$.

Let $\{ \widetilde{B}^* \}$ be the set of all the unit lattice cubes that 
intersect $V$, and $X= \cup \, \widetilde{B}^*$. Also, let $B_r$ be a ball 
in $\mbb R^n$ of radius $r \geq R^\epsilon$. We want to estimate the number 
of the cubes $\widetilde{B}^*$ that intersect $B_r$. In order to do this, 
we need to estimate the number of balls $B(c(\widetilde{B}),2R^\epsilon)$ 
that intersect $B_r$.

We have
\begin{displaymath}
\int_{B(c(\widetilde{B}),R^\epsilon)} H(x) dx 
\geq \int_{\widetilde{B}} H(x) dx = v(\widetilde{B}) A_\alpha(H)
\geq n^{\alpha/2} 2^{k_0-1} A_\alpha(H),
\end{displaymath}
so (using (\ref{bdoverlap}))
\begin{eqnarray*}
\lefteqn{\Big( \#\{ \widetilde{B} \in V_{k_0}: 
         B(c(\widetilde{B}),2R^\epsilon) \cap B_r \not= \emptyset \} \Big) 
         2^{k_0} A_\alpha(H)} \\
& & \lct \; R^{n \epsilon} \int_{B_{3r}} H(x) dx 
 \; \leq \; R^{n \epsilon} A_\alpha(H) (3r)^\alpha,
\end{eqnarray*}
and so
\begin{displaymath}
\#\{ \widetilde{B} \in V_{k_0}: B(c(\widetilde{B}),2R^\epsilon) \cap B_r
\not= \emptyset \} \lct R^{n \epsilon} 2^{-k_0} r^\alpha.
\end{displaymath}
Thus
\begin{displaymath}
\#\{ \widetilde{B}^* : \widetilde{B}^* \subset B_r \} 
\lct R^{2n\epsilon} 2^{-k_0} r^\alpha.
\end{displaymath}
Therefore, we can apply Theorem \ref{highfracdim} with 
$\gamma \sim R^{2n\epsilon} 2^{-k_0}$ to get
\begin{displaymath}
\int_V |Ef(x)|^2 dx \leq \int_X |Ef(x)|^2 dx 
\lct R^{6\epsilon} (2^{-k_0})^{2/n} R^{\alpha/n} \| f \|_{L^2(S)}^2,
\end{displaymath}
which, combined with (\ref{pigeonholing}) and (\ref{localstar}), now tells 
us that
\begin{eqnarray*}
\int_{B(0,R)} |Ef(x)|^2 dx 
& \lct & R^{(n+6)\epsilon} (2^{k_0})^{1-(2/n)} A_\alpha(H) R^{\alpha/n} 
         \| f \|_{L^2(S)}^2 \\
& \lct & R^{(n+6)\epsilon} A_\alpha(H) R^{\alpha/n} \| f \|_{L^2(S)}^2.
\end{eqnarray*}
We note that for the last inequality, we need the fact $2^{k_0}$ is raised 
to a non-negative power, which is a consequence of the fact that the 
exponent of $\gamma$ in the estimate of Theorem \ref{highfracdim} is less 
than or equal to $1/2$, which is also the case in the estimate of Corollary 
\ref{coromain1}.
\end{proof}

\section{Proof of the weighted H\"{o}lder-type inequality and its 
corollary} 

In this section we prove Theorem \ref{whtineq} and Corollary 
\ref{indconclusion}.

\begin{proof}[Proof of Theorem \ref{whtineq}]
For $N \in \mbb N$, we let $\chi_N$ be the characteristic function of the
set
\begin{displaymath}
B(0,N) \cap \{ x \in \mbb R^n : F(x) \leq N \},
\end{displaymath}	
and we define the function $F_N$ by $F_N= \chi_N F$. Clearly,
\begin{displaymath}
\int F_N(x) H(x) dx \leq N \int_{B(0,N)} H(x) dx 
\leq A_\alpha(H) N^{1+\alpha}
\end{displaymath}
for all weights $H$ on $\mbb R^n$ of dimension $\alpha$. Letting 
$\beta_0=1$ and $C_0=N^{1+\alpha}$, this becomes
\begin{equation}
\label{que0}
\int F_N(x)^{\beta_0} H(x) dx \leq C_0 A_\alpha(H).
\end{equation}

Let $H$ be a weight on $\mbb R^n$ of dimension $\alpha$.

If $p > 1$ and $B_R$ is a ball in $\mbb R^n$ of radius $R \geq 1$, then 
(\ref{que0}) and H\"{o}lder's inequality tell us that
\begin{eqnarray*}
\int_{B_R} F_N(x)^{\beta_0/p} H(x) dx 
& \leq & \Big( \int_{B_R} F_N(x)^{\beta_0} H(x) dx \Big)^{1/p}
         \Big( \int_{B_R} H(x) dx \Big)^{1/p'} \\
& \leq & \Big( C_0 A_\alpha(H) \Big)^{1/p} 
         \Big( A_\alpha(H) R^\alpha \Big)^{1/p'} \\
&   =  & C_0^{1/p} A_\alpha(H) R^{\alpha/p'},
\end{eqnarray*}
where $p'$ is the exponent conjugate to $p$.

We now choose $p$ such that $\alpha/p'=\beta$, i.e.\ 
$p=\alpha/(\alpha-\beta)$, to conclude that the function 
${\mathcal H}(x) := N^{-1/p} F_N(x)^{\beta_0/p} H(x)$ is a weight on 
$\mbb R^n$ of dimension $\beta$. Moreover,
\begin{displaymath}
A_\beta({\mathcal H}) \leq N^{-1/p} C_0^{1/p} A_\alpha(H).
\end{displaymath}
Therefore,
\begin{eqnarray*}
\int F_N(x)^\beta N^{-1/p} F_N(x)^{\beta_0/p} H(x) dx 
&   =  & \int F_N(x)^\beta {\mathcal H}(x)dx \\
& \leq & (M_\beta F_N)^\beta A_\beta({\mathcal H}) \\
& \leq & (M_\beta F)^\beta N^{-1/p} C_0^{1/p} A_\alpha(H),
\end{eqnarray*}
where we have used the fact that $F_N \leq F$ to conclude that
$M_\beta F_N \leq M_\beta F$. Letting $M=(M_\beta F)^\beta$,
$\beta_1=\beta+(\beta_0/p)$, and $C_1 = M C_0^{1/p}$, this becomes
\begin{equation}
\label{que1}
\int F_N(x)^{\beta_1} H(x) dx \leq C_1 A_\alpha(H).
\end{equation}

Iterating the above procedure starting from (\ref{que1}) instead of 
(\ref{que0}), we arrive at
\begin{displaymath}
\int F_N(x)^{\beta_2} H(x) dx \leq C_2 A_\alpha(H)	
\end{displaymath}
with $\beta_2=\beta+(\beta_1/p)$ and $C_2= M C_1^{1/p}$. Proceeding in this 
fashion and using mathematical induction, we obtain two sequences 
$\{ \beta_k \}$ and $\{ C_k \}$ of non-negative numbers such that
\begin{equation}
\label{quek}
\int F_N(x)^{\beta_k} H(x) dx \leq C_k A_\alpha(H),
\end{equation}
$\beta_k=\beta+(\beta_{k-1}/p)$, and 
$C_k = M C_{k-1}^{1/p}$.

Now
\begin{eqnarray*}
\lefteqn{\beta_k = \beta + \frac{\beta_{k-1}}{p} 
             = \beta + \Big( \beta + \frac{\beta_{k-2}}{p} \Big) \frac{1}{p}
             = \beta + \frac{\beta}{p} + \frac{\beta_{k-2}}{p^2}
             = \beta + \frac{\beta}{p} 
               + \Big( \beta + \frac{\beta_{k-3}}{p} \Big) \frac{1}{p^2}} \\
& & \mbox{} \hspace{-0.3in} 
    = \beta + \frac{\beta}{p} + \frac{\beta}{p^2} + \frac{\beta_{k-3}}{p^3}
    = \beta + \frac{\beta}{p} +  \frac{\beta}{p^2} + \cdots + 
      \frac{\beta}{p^{k-1}} + \frac{\beta_{k-k}}{p^k}
    = \beta \frac{1-(1/p)^k}{1-(1/p)} + \frac{\beta_0}{p^k}
\end{eqnarray*}
and
\begin{eqnarray*}
\lefteqn{C_k = M C_{k-1}^{1/p} = M \Big( M C_{k-2}^{1/p} \Big)^{1/p} 
= M^{1+(1/p)} C_{k-2}^{(1/p)^2}} \\
& & \mbox{} \hspace{-0.3in}
    = M^{1+(1/p)} \Big( M C_{k-3}^{1/p} \Big)^{(1/p)^2}
    = M^{1+(1/p)+(1/p)^2} C_{k-3}^{(1/p)^3} \\
& & \mbox{} \hspace{-0.3in}
    = M^{1+(1/p)+(1/p)^2+ \cdots + (1/p)^{k-1}} C_{k-k}^{(1/p)^k} 
    = M^{(1-(1/p)^k)/(1-(1/p))} C_0^{(1/p)^k},
\end{eqnarray*}
so (recalling that $p=\alpha/(\alpha-\beta)$)
\begin{displaymath}
\lim_{k \to \infty} \beta_k = \frac{\beta}{1-(1/p)} = \alpha
\hspace{0.5in} \mbox{and} \hspace{0.5in}
\lim_{k \to \infty} C_k = M^{\alpha/\beta}=(M_\beta F)^\alpha.
\end{displaymath}
Therefore, letting $k \to \infty$ in (\ref{quek}) and using Fatou's lemma, 
we arrive at
\begin{displaymath}
\int F_N(x)^\alpha H(x) dx \leq (M_\beta F)^\alpha A_\alpha(H).
\end{displaymath}

Since $F_N \to F$ pointwise on $\mbb R^n$ as $N \to \infty$, a second 
application of Fatou's lemma gives us
\begin{displaymath}
\int F(x)^\alpha H(x) dx \leq (M_\beta F)^\alpha A_\alpha(H).
\end{displaymath}
Since the last inequality holds for all weights $H$ on $\mbb R^n$ of
dimension $\alpha$, it follows that $M_\alpha F \leq M_\beta F$.
\end{proof}

\begin{proof}[Proof of Corollary \ref{indconclusion}]
We will only prove the inequality concerning $Q_{\rm loc}(\alpha,p)$. The
proof for $Q(\alpha,p)$ is similar and a little easier.

We may assume $Q_{\rm loc}(\beta,p) < \infty$ (otherwise, there is nothing 
to prove). Let $q > Q_{\rm loc}(\beta,p)$. Then by the definition of
$Q_{\rm loc}(\beta,p)$, to every $\epsilon > 0$ there is a constant 
$C_\epsilon$ such that
\begin{displaymath}
\int_{B(0,R)} |Ef(x)|^q {\mathcal H}(x) dx \leq C_\epsilon R^\epsilon
A_\beta({\mathcal H}) \| f \|_{L^p(S)}^{q}
\end{displaymath}
for all functions $f \in L^p(S)$ and weights ${\mathcal H}$ on $\mbb R^n$ of
dimension $\beta$. Letting $F= \chi_{B(0,R)}|Ef|^{q/\beta}$, this implies
\begin{displaymath}
M_\beta(F^\beta) 
\leq \Big( C_\epsilon R^\epsilon \| f \|_{L^p(S)}^{q} \Big)^{1/\beta}.
\end{displaymath}
Applying Theorem \ref{whtineq}, we get
\begin{displaymath}
M_\alpha(F^\beta) 
\leq \Big( C_\epsilon R^\epsilon \| f \|_{L^p(S)}^{q} \Big)^{1/\beta}.
\end{displaymath}
Therefore,
\begin{displaymath}
\Big( \frac{1}{A_\alpha(H)} \int_{B(0,R)} |Ef(x)|^{(\alpha/\beta)q} H(x) dx 
\Big)^{1/\alpha}
\leq \Big( C_\epsilon R^\epsilon \| f \|_{L^p(S)}^{q} \Big)^{1/\beta}
\end{displaymath}
for all functions $f \in L^p(S)$ and weights $H$ on $\mbb R^n$ of dimension
$\alpha$. 

Recalling the definition of $Q_{\rm loc}(\alpha,p)$, we now have 
$(\alpha/\beta)q \geq Q_{\rm loc}(\alpha,p)$. Since this inequality is true
for all $q > Q_{\rm loc}(\beta,p)$, it follows that
\begin{displaymath}
(\alpha/\beta) Q_{\rm loc}(\beta,p) \geq Q_{\rm loc}(\alpha,p),
\end{displaymath}
as desired.
\end{proof}

\section{Estimates in the regimes $0 < \alpha < (n-1)/2$ and 
         $(n+1)/2 < \alpha < n$}

We start by proving two $L^2$-based weighted restriction estimates. The 
first estimate, which is part (i) of Proposition \ref{base}, proves Theorem 
\ref{main1} in the regime $0 < \alpha < (n-1)/2$, and, as discussed in the 
Introduction, is the base for proving the theorem in the other two regimes 
$(n-1)/2 \leq \alpha \leq n/2$ and $n/2 < \alpha \leq (n+1)/2$. The second 
estimate, which is part (ii) of Proposition \ref{base}, will be one of the
main components of the proof of Theorem \ref{main3}. The work we do in this
section is based on ideas from \cite{jb:besicovitch}.

\begin{prop}
\label{base}
Suppose $S$ is a smooth compact hypersurface in $R^n$ with a nowhere 
vanishing Gaussian curvature, and $H$ is a weight on $\mbb R^n$ of 
dimension $0 < \alpha < (n-1)/2$. Then:
 
{\rm (i)} To every exponent $q > 2$ there is a constant $C_q$, which does 
not depend on $H$, such that
\begin{displaymath}
\| E f \|_{L^q(Hdx)} \leq C_q A_\alpha(H)^{1/q} \| f \|_{L^2(S)}
\end{displaymath}
for all $f \in L^2(S)$.

{\rm (ii)} To every exponent $q > (n+1-2\alpha)/(n-2\alpha)$ there is a
constant $\bar{C}_q$, which does not depend on $H$, such that
\begin{displaymath}
\| E f \|_{L^q(Hdx)} \leq \bar{C}_q A_\alpha(H)^{1/(q(n-2\alpha))} 
\| H \|_{L^2(\mbb R^n)}^{(n-1-2\alpha)/(q(n-2\alpha))} \| f \|_{L^2(S)}
\end{displaymath}
for all $f \in L^2(S)$.
\end{prop}

\begin{proof}
We may assume that $H \in L^1(\mbb R^n)$. (Otherwise, we multiply $H$ by the
characteristic function of the ball $B(0,R)$, obtain an estimate that is
uniform in $R$, and then send $R$ to infinity using the fact that
$A_\alpha(\chi_{B(0,R)} H) \leq A_\alpha(H)$.) We define the measure 
$\mu$ on $\mbb R^n$ by $d\mu= Hdx$. 

Let $f \in L^2(S)$. We need to estimate $\| E f \|_{L^q(\mu)}$. We write
\begin{equation}
\label{bloc1}
\| E f \|_{L^q(\mu)}^q = \int_0^{\| f \|_{L^1(S)}} q \lambda^{q-1} 
  \mu \big( \big\{ |E f| \geq \lambda \big\} \big) d\lambda.
\end{equation}
The set $\{ |E f| \geq \lambda \}$ is contained in
\begin{displaymath}
\big\{ (\mbox{Re} \, E f)_+ \geq \frac{\lambda}{4} \big\} \cup 
\big\{ (\mbox{Re} \, E f)_- \geq \frac{\lambda}{4} \big\} \cup
\big\{ (\mbox{Im} \, E f)_+ \geq \frac{\lambda}{4} \big\} \cup 
\big\{ (\mbox{Im} \, E f)_- \geq \frac{\lambda}{4} \big\},
\end{displaymath}
where $(\mbox{Re} \, E f)_+$ and $(\mbox{Re} \, E f)_-$ are, respectively, 
the positive and negative parts of $\mbox{Re} \, Ef$; and similarly for 
$\mbox{Im} \, E f$. Therefore, it is enough to estimate the $\mu$-measure 
of the set $\{ (\mbox{Re} \, E f)_+ \geq \lambda/4 \}$. We denote this set 
by $G$, and we observe that
\begin{displaymath}
G = \big\{ \mbox{Re} \, E f \geq \frac{\lambda}{4} \big\}
\end{displaymath}
for $\lambda > 0$. So
\begin{displaymath}
\frac{\lambda}{4} \mu(G) 
\leq \int_G (\mbox{Re} \, E f) d\mu
= \mbox{Re} \int_G \, E f \, d\mu
= \mbox{Re} \int \chi_G \, \widehat{fd\sigma} \, d\mu
= \mbox{Re} \int \widehat{\chi_G d\mu} \, f d\sigma,
\end{displaymath}
and so (by Cauchy-Schwarz)
\begin{displaymath}
\lambda^2 \mu(G)^2 \leq 
16 \| f \|_{L^2(S)}^2 \| \widehat{\chi_G d\mu} \|_{L^2(S)}^2.
\end{displaymath}
By the duality relation of the Fourier transform, we have
\begin{displaymath}
\| \widehat{\chi_G d\mu} \|_{L^2(S)}^2
= \int \widehat{\chi_G d\mu} \; \overline{\widehat{\chi_G d\mu}}  \, d\sigma
= \int \Big( \overline{\widehat{\chi_G d\mu}} \, d\sigma \widehat{\Big)} \,
  \chi_G d\mu
= \int \big( (\chi_G d\mu) \ast \widehat{\sigma} \big) \chi_G d\mu,
\end{displaymath}
so
\begin{equation}
\label{bloc2}
\lambda^2 \mu(G)^2 \leq 16 \| f \|_{L^2(S)}^2 
\int \big( (\chi_G d\mu) \ast \widehat{\sigma} \big) \chi_G d\mu.
\end{equation}

The next step is to invoke the decay estimate we have on $\widehat{\sigma}$:
$|\widehat{\sigma}(\xi)| \lct |\xi|^{-(n-1)/2}$ for all $|\xi| \geq 1$
(which is a consequence of the nowhere vanishing Gaussian curvature
assumption on the surface $S$), as well as the dimensionality of the 
measure $\mu$:
\begin{displaymath}
\mu(B(x_0,R)) = \int_{B(x_0,R)} H(x) dx \leq A_\alpha(H) R^\alpha	
\end{displaymath}
for all $x_0 \in \mbb R^n$ and $R \geq 1$. 

We let $\psi_0$ be a $C_0^\infty$ function on $\mbb R^n$ satisfying 
$0 \leq \psi_0 \leq 1$, $\psi_0=1$ on $B(0,1)$, and $\psi_0=0$ outside 
$B(0,2)$. Also, for $l \in \mbb N$, we define 
$\psi_l(x)=\psi_0(x/2^l)-\psi_0(x/2^{l-1})$. Then $\psi_l$ is supported in 
the ring $2^{l-1} \leq |x| \leq 2^{l+1}$, and
\begin{displaymath}
(\chi_G d\mu) \ast \widehat{\sigma} = 
\sum_{k=0}^\infty (\chi_G d\mu) \ast (\psi_k \, \widehat{\sigma}).
\end{displaymath}
Since $|\psi_0 \, \widehat{\sigma}| \lct 1$ and 
$|\psi_l \, \widehat{\sigma}| \lct 2^{-(l-1)(n-1)/2}$, we have
\begin{eqnarray}
\label{block1}
|(\chi_G d\mu) \ast (\psi_k \, \widehat{\sigma})(x)| 
& \leq & \int |\psi_k(x-y) \, \widehat{\sigma}(x-y)| \chi_G(y) d\mu(y) 
         \nonumber \\
& \lct & 2^{-k(n-1)/2} \int \chi_{B(x,2^{k+1})}(y) \chi_G(y) d\mu(y) 
         \nonumber \\
& \lct & 2^{-k(n-1)/2} \mu \big( B(x,2^{k+1}) \big) \nonumber \\
& \lct & A_\alpha(H) 2^{-k(n-1-2\alpha)/2},
\end{eqnarray}
and since $\alpha < (n-1)/2$, it follows that
\begin{displaymath}
|(\chi_G d\mu) \ast \widehat{\sigma}(x)| \lct 
\sum_{k=0}^\infty A_\alpha(H) 2^{-k(n-1-2\alpha)/2} \lct A_\alpha(H)
\end{displaymath}
for all $x \in \mbb R^n$.

Returning to (\ref{bloc2}), we now have
$\lambda^2 \mu(G)^2 \lct \| f \|_{L^2(S)}^2 A_\alpha(H) \mu(G)$. Therefore, 
by (\ref{bloc1}),
\begin{displaymath}
\| E f \|_{L^q(\mu)}^q \lct  A_\alpha(H) \| f \|_{L^2(S)}^2 
\int_0^{\| f \|_{L^1(S)}}  \lambda^{q-3} d\lambda
\lct  A_\alpha(H) \| f \|_{L^2(S)}^q
\end{displaymath}
provided $q >2$. This proves part (i).

We note that in proving part (i) we did not use the dimensionality of the 
measure $\sigma$: $\sigma(B(x_0,r)) \lct r^{n-1}$ for all 
$x_0 \in \mbb R^n$ and $r > 0$. But the dimensionality of $\sigma$ will be 
used in proving part (ii) in the following form:
\begin{equation}
\label{dimsigma}
\| \widehat{\psi_k} \ast \sigma \|_{L^\infty} \lct 2^k
\end{equation}
for $k=0,1,2, \ldots$.

The inequality (\ref{block1}) is a bound on 
$\| (\chi_G d\mu) \ast (\psi_k \, \widehat{\sigma}) \|_{L^\infty}$, which
implies that
\begin{displaymath}
\int |(\chi_G d\mu) \ast (\psi_k \, \widehat{\sigma})| \chi_G d\mu
\lct 2^{-k(n-1-2\alpha)/2} A_\alpha(H) \mu(G).
\end{displaymath}
We now derive a second bound on
$\int |(\chi_G d\mu) \ast (\psi_k \, \widehat{\sigma})| \chi_G d\mu$. By 
Plancherel and (\ref{dimsigma}),
\begin{displaymath}
\| (\chi_G d\mu) \ast (\psi_k \, \widehat{\sigma}) \|_{L^2}
\lct 2^k \| \widehat{\chi_G H} \|_{L^2} = 2^k \| \chi_G H \|_{L^2},
\end{displaymath}
so (by Cauchy-Schwarz)
\begin{equation}
\label{closely}
\int |(\chi_G d\mu) \ast (\psi_k \, \widehat{\sigma})| \chi_G d\mu
\lct 2^k \| \chi_G H \|_{L^2}^2.
\end{equation}
Thus
\begin{displaymath}
\int |(\chi_G d\mu) \ast (\psi_k \, \widehat{\sigma})| \chi_G d\mu
\lct \min \big[ 2^k \| H \|_{L^2}^2, 
                  2^{-k(n-1-2\alpha)/2} A_\alpha(H) \mu(G) \big]
\end{displaymath}
for $k=0,1,2, \ldots$, where we have used the fact that 
$\| \chi_G H \|_{L^2} \leq \| H \|_{L^2}$.

Returning to (\ref{bloc2}), we now have
\begin{displaymath}
\lambda^2 \mu(G)^2 \lct \| f \|_{L^2(S)}^2 
\Big( \sum_{k=0}^{k_0} 2^k \| H \|_{L^2}^2 
      + \sum_{k=k_0}^\infty 2^{-k(n-1-2\alpha)/2} A_\alpha(H) \mu(G)
\Big),
\end{displaymath}
where $k_0$ is a positive integer that satisfies
\begin{displaymath}
2^{k_0} \sim \Big( \frac{A_\alpha(H) \mu(G)}{\| H \|_{L^2}^2} 
             \Big)^{2/(n+1-2\alpha)}.
\end{displaymath}
Since $(n-1-2\alpha)/2 > 0$, the geometric series converges giving
\begin{displaymath}
\lambda^2 \mu(G)^2 \lct \| f \|_{L^2(S)}^2 
\Big( A_\alpha(H) \mu(G) \Big)^{2/(n+1-2\alpha)}
\| H \|_{L^2}^{2(n-1-2\alpha)/(n+1-2\alpha)},
\end{displaymath}
which in turn implies that
\begin{displaymath}
\mu(G) \lct A_\alpha(H)^{1/(n-2\alpha)} 
\| H \|_{L^2}^{(n-1-2\alpha)/(n-2\alpha)} 
\big( \lambda^{-1} \| f \|_{L^2(S)} \big)^{(n+1-2\alpha)/(n-2\alpha)}.
\end{displaymath}
Inserting this bound on $\mu(G)$ into (\ref{bloc1}), we obtain
\begin{eqnarray*}
\| E f \|_{L^q(\mu)}^q
& \lct & A_\alpha(H)^{1/(n-2\alpha)} 
         \| H \|_{L^2}^{(n-1-2\alpha)/(n-2\alpha)} 
         \| f \|_{L^2(S)}^{(n+1-2\alpha)/(n-2\alpha)} \\
&      & \times
         \int_0^{\| f \|_{L^1(S)}} \lambda^{q-1-(n+1-2\alpha)/(n-2\alpha)} 
         d\lambda \\
& \lct & A_\alpha(H)^{1/(n-2\alpha)} 
         \| H \|_{L^2}^{(n-1-2\alpha)/(n-2\alpha)} \| f \|_{L^2(S)}^q
\end{eqnarray*}
provided $q > (n+1-2\alpha)/(n-2\alpha)$, which proves part (ii).
\end{proof}

Readers who are familiar with Bourgain's paper \cite{jb:besicovitch} will
realize that we can follow that paper more closely by inserting a favorable 
local restriction estimate in the inequality immediately preceding
(\ref{block1}). The argument will then proceed as follows.

Suppose $1 \leq \alpha < n$. The last inequality before (\ref{block1}) says
\begin{displaymath}
\int \chi_{B(x,2^{k+1})}(y) \chi_G(y) d\mu(y) \lct \mu(B(x,2^{k+1})).
\end{displaymath}
We replace this by
\begin{eqnarray*}
\int_{B(x,2^{k+1})} \chi_G(y) d\mu(y) 
& \lct & \lambda^{-2} \int_{B(x,2^{k+1})} |Ef(y)|^2 H(y) dy \\
& \lct & \lambda^{-2} 2^{k\epsilon} A_\alpha(H) 2^{k\alpha/n} 
         \| f \|_{L^2(S)}^2,
\end{eqnarray*}
where on the first line we used the fact that 
$\chi_G \leq 4 \lambda^{-1} |Ef|$, and on the second line we used the
Du and Zhang estimate from Corollary \ref{highfracdimcoro}. Inequality
(\ref{block1}) becomes
\begin{displaymath}
|(\chi_G d\mu) \ast (\psi_k \widehat{\sigma})(x)| 
\lct \lambda^{-2} A_\alpha(H) 
     2^{-k \big( (n-1)/2-(\alpha/n)-\epsilon \big)}	\| f \|_{L^2(S)}^2,
\end{displaymath}
so that
\begin{displaymath}
\int |(\chi_G d\mu) \ast (\psi_k \, \widehat{\sigma})| \chi_G d\mu
\lct \lambda^{-2} A_\alpha(H) 
     2^{-k \big( (n-1)/2-(\alpha/n)-\epsilon \big)} \| f \|_{L^2(S)}^2 
     \mu(G).
\end{displaymath}
Combining this inequality with (\ref{closely}), we arrive at
\begin{eqnarray*}
\lefteqn{\int |(\chi_G d\mu) \ast (\psi_k \, \widehat{\sigma})| \chi_G 
         d\mu} \\
& \lct & \!\!\! \min \Big[ 2^k \| \chi_G H \|_{L^2}^2, \lambda^{-2} 
         A_\alpha(H) 2^{-k \big( (n-1)/2-(\alpha/n)-\epsilon \big)} 
         \| f \|_{L^2(S)}^2 \mu(G) \Big] \\
& \leq & \!\!\! \mu(G) \, \min \Big[ 2^k, \lambda^{-2} A_\alpha(H) 
         2^{-k \big( (n-1)/2-(\alpha/n)-\epsilon \big)} 
         \| f \|_{L^2(S)}^2 \Big]
\end{eqnarray*}
for $k=0,1,2, \ldots$, where we have used the fact that 
$\| \chi_G H \|_{L^2}^2 = \int_G H(x)^2 dx \leq \int_G H(x) dx = \mu(G)$.

Returning to (\ref{bloc2}), we now have
\begin{displaymath}
\lambda^2 \mu(G) \lct \| f \|_{L^2(S)}^2
\Big( \sum_{k=0}^{k_0} 2^k 
      + \sum_{k=k_0}^\infty \lambda^{-2} A_\alpha(H) 	
        2^{-k \big( (n-1)/2-(\alpha/n)-\epsilon \big)} 
       \| f \|_{L^2(S)}^2 \Big),
\end{displaymath}
where $k_0$ is a positive integer that satisfies
\begin{displaymath}
2^{k_0} \sim \Big( \lambda^{-2} A_\alpha(H) \| f \|_{L^2(S)}^2 
             \Big)^{1/((n+1)/2-(\alpha/n)-\epsilon)}.
\end{displaymath}
For the geometric series to converge, we must have 
$(n-1)/2-(\alpha/n)-\epsilon > 0$, i.e.
\begin{displaymath}
\alpha < \frac{n(n-1)}{2} - n \epsilon.	
\end{displaymath}
This is possible if $\alpha < n(n-1)/2$. In the plane, this condition 
becomes $\alpha < 1$. But for Corollary \ref{highfracdimcoro} to hold, we 
need $\alpha \geq 1$, so for the rest of this argument we must work in 
$\mbb R^n$ with $n \geq 3$. So, choosing $\epsilon$ sufficiently small, we
get
\begin{displaymath}
\lambda^2 \mu(G) \lct \| f \|_{L^2(S)}^2
\Big( \lambda^{-2} A_\alpha(H) \| f \|_{L^2(S)}^2 
\Big)^{1/((n+1)/2-(\alpha/n)-\epsilon)},
\end{displaymath}
and so
\begin{displaymath}
\mu(G) \lct \lambda^{-2 q_\epsilon} \| f \|_{L^2(S)}^{2q_\epsilon}
A_\alpha(H)^{1/((n+1)/2-(\alpha/n)-\epsilon)},
\end{displaymath}
where 
$q_\epsilon=((n+3)/2-(\alpha/n)-\epsilon)/((n+1)/2-(\alpha/n)-\epsilon)$.

Inserting the bound we now have on $\mu(G)$ into (\ref{bloc1}), we obtain
\begin{eqnarray*}
\| E f \|_{L^q(\mu)}^q
& \lct & \| f \|_{L^2(S)}^{2q_\epsilon} 
         A_\alpha(H)^{1/((n+1)/2-(\alpha/n)-\epsilon)}
         \int_0^{\| f \|_{L^1(S)}} \lambda^{q-1-2q_\epsilon} d\lambda \\
& \lct & A_\alpha(H)^{1/((n+1)/2-(\alpha/n)-\epsilon)} \| f \|_{L^2(S)}^q
\end{eqnarray*}
provided $q > 2q_\epsilon$. Since
\begin{displaymath}
\lim_{\epsilon \to 0} 2q_\epsilon
= 2 \frac{n^2+3n-2\alpha}{n^2+n-2\alpha},
\end{displaymath}
we obtain the following result.

\begin{prop}
\label{simbase}
Suppose that $n \geq 3$, $1 \leq \alpha < n$, and $S$ is a compact 
$C^\infty$ hypersurface in $\mbb R^n$ with a strictly positive second 
fundamental form. Then to every exponent 
$q > 2(n^2+3n-2\alpha)/(n^2+n-2\alpha)$ there is a constant $c_q$ 
satisfying $c_q < ((n-1)/2)-(\alpha/n)$ such that the following holds: if 
$0< \epsilon < c_q$, then
\begin{displaymath}
\int |Ef(x)|^q H(x) dx 
\lct A_\alpha(H)^{1/((n+1)/2-(\alpha/n)-\epsilon)} \| f \|_{L^2(S)}^q
\end{displaymath}	
for all functions $f \in L^2(S)$ and weights $H$ on $\mbb R^n$ of dimension
$\alpha$.
\end{prop}

We note that $2(n^2+3n-2\alpha)/(n^2+n-2\alpha)=2(n+1)/(n-1)$ if 
$\alpha=n$, so Proposition \ref{simbase} improves on Tomas-Stein for all
$1 \leq \alpha < n$. But
\begin{displaymath}
2 \frac{n^2+3n-2\alpha}{n^2+n-2\alpha} > \frac{2n}{n-1}
\hspace{0.25in} \mbox{for} \hspace{0.25in} n \geq 3 \mbox{ and }
0 < \alpha < \frac{n(n+1)}{2},
\end{displaymath}
so Theorem \ref{main1} gives a far better result for $0 < \alpha \leq n/2$.
In fact, the range of $q$ in Theorem \ref{main1} is better than that in
Proposition \ref{simbase} for $0 < \alpha \leq \alpha_n$, where $\alpha_n$ 
is the smaller of the two solutions of the equation
\begin{displaymath}
2 \frac{n^2+3n-2\alpha}{n^2+n-2\alpha}
= \frac{2n}{n-1} + 2 - \frac{n}{\alpha}.
\end{displaymath}
Solving this equation, we see that
\begin{displaymath}
\alpha_n= \frac{n^2+1-\sqrt{n^4-4n^3+2n^2+4n+1}}{4} = \frac{n+1}{2}.
\end{displaymath}

\underline{\bf Example.} Suppose $n \geq 3$ and $Z$ is the zero set of a
polynomial $P$ on $\mbb R^n$ of degree $D \geq 1$. Also, suppose $N_\rho(Z)$
is the $\rho$-neighborhood of $Z$ and $H$ is the characteristic function of
$N_\rho(Z)$. As we saw in the first example of Section 3, $H$ is a weight on
$\mbb R^n$ of dimension $n-1$ with $A_{n-1}(H) \leq C_n D \rho$. So we can
apply Proposition \ref{simbase} with $\alpha=n-1$.

The exponent of $A_{n-1}(H)$ in Proposition \ref{simbase} is
\begin{displaymath}
\Big( \frac{n+1}{2} - \frac{n-1}{n} - \epsilon \Big)^{-1} \leq \frac{2}{n-1}
\end{displaymath}
provided $\epsilon < 1/n$. Therefore, we have the estimate
\begin{displaymath}
\int_{N_\rho(Z)} |Ef(x)|^q dx 
\lct (D \rho)^{2/(n-1)} \| f \|_{L^2(S)}^q \hspace{0.5in}
\Big( q > 2 \frac{n^2+n+2}{n^2-n+2} \Big)
\end{displaymath}
for all $\rho \geq D^{-1}$. One interesting aspect of this estimate is that
it holds beyond the $(2n+2)/(n-1)$ exponent of Tomas-Stein, another 
interesting aspect is that the exponent of $\rho$ goes to zero as 
$n \to \infty$.

\section{Proof of Theorem \ref{main1}}

Having discussed in detail the Du and Zhang fractal restriction theorem, 
proven the weighted H\"{o}lder-type inequality and its corollary, and 
established a good restriction estimate in fractal dimensions 
$0 < \alpha < (n-1)/2$, we are now ready to put all those components 
together and prove Theorem \ref{main1}.

\begin{proof}[Proof of Theorem \ref{main1}]
Let $Q(\alpha,2)$ be the quantity defined right before the statement of
Corollary \ref{indconclusion}. We need to show that
\begin{equation}
\label{main1pf1}
Q(\alpha,2) \leq \left\{ \begin{array}{ll}
2 & \mbox{ if \, $0 < \alpha < (n-1)/2$,} \\
4\alpha/(n-1) & \mbox{ if \, $(n-1)/2 \leq \alpha \leq n/2$,} \\
2 \alpha + 2 & \mbox{ if \, $n=2$ and $1 < \alpha < 2$,} \\
(2n/(n-1)) +2- (n/\alpha) & \mbox{ if \, $n \geq 3$ and
                                  $n/2 < \alpha \leq n$.}
\end{array} \right.
\end{equation}
(See the statement of Theorem \ref{main1} and the paragraph immediately 
following it.)

Part (i) of Proposition \ref{base} immediately gives the inequality on
the first line of (\ref{main1pf1}). Then, applying Theorem 
\ref{indconclusion} with $0 < \beta < (n-1)/2 \leq \alpha \leq n/2$, we get
\begin{displaymath}
\frac{Q(\alpha,2)}{\alpha} \leq \frac{Q(\beta,2)}{\beta} 
\leq \frac{2}{\beta}.
\end{displaymath}
Therefore (letting $\beta \to (n-1)/2$), $Q(\alpha,2) \leq 4 \alpha/(n-1)$. 

It remains to prove the last two lines of (\ref{main1pf1}). For this we 
need Corollary \ref{highfracdimcoro}.

Suppose $n=2$ and $1 < \alpha < 2$. Also, let $\epsilon > 0$ and 
$f \in L^2(S)$. Then Corollary \ref{highfracdimcoro} tells us that
\begin{displaymath}
\int_{B_R} |Ef(x)|^2 H(x) dx 
\leq C_\epsilon R^\epsilon A_\alpha(H) R^{\alpha/2} \| f \|_{L^2(S)}^2
\end{displaymath}
for all balls $B_R \subset \mbb R^2$ of radius $R \geq 1$. Thus the 
function 
\begin{displaymath}
{\mathcal H}(x) := \| f \|_{L^1(S)}^{-2} |Ef(x)|^2 H(x) 
\end{displaymath}
is a weight on $\mbb R^2$ of dimension $\alpha'=(\alpha/2)+\epsilon$ and 
with
\begin{displaymath} 
A_{\alpha'}({\mathcal H}) 
\lct A_\alpha(H) \| f \|_{L^1(S)}^{-2} \| f \|_{L^2(S)}^2.
\end{displaymath}
Since $1 < \alpha < 2$, we have $1/2 < \alpha' < 1$ provided $\epsilon$ is 
sufficiently small. So $Q(\alpha',2) \leq 4\alpha'$, and so
\begin{displaymath}
\int |Eg(x)|^{q'} {\mathcal H}(x) dx \lct A_{\alpha'}({\mathcal H}) 
\| g \|_{L^2(S)}^{q'}
\end{displaymath}
for all $g \in L^2(S)$ provided $q' > 4 \alpha' = 2 \alpha + 4\epsilon$.
Replacing ${\mathcal H}$ by $\| f \|_{L^1(S)}^{-2} |Ef|^2 H$, plugging $f$ 
for $g$, and choosing $\epsilon$ to be sufficiently small, the last 
estimate becomes
\begin{displaymath}
\int |Ef(x)|^q H(x) dx \lct A_\alpha(H) \| f \|_{L^2(S)}^q
\end{displaymath}
for $q > 2\alpha +2$, which proves the inequality on the line next to the 
last in (\ref{main1pf1}).

Now suppose $n \geq 3$ and $n/2 < \alpha \leq n$. Also, let $\epsilon > 0$,
$0 < p \leq 2$, and $f \in L^2(S)$. Then Corollary \ref{highfracdimcoro} 
and H\"{o}lder's inequality tell us that
\begin{eqnarray*}
\int_{B_R} |Ef(x)|^p H(x) dx \!\!\!
& \leq & \!\!\! \Big( C_\epsilon R^\epsilon A_\alpha(H) R^{\alpha/n} 
               \| f \|_{L^2(S)}^2 \Big)^{p/2} 
         \Big( \int_{B_R} H(x) dx \Big)^{1-(p/2)} \\
& \leq & \!\!\! C_\epsilon^{p/2} R^{p\epsilon/2} A_\alpha(H) 
         \| f \|_{L^2(S)}^p R^{(1-((n-1)p/(2n))\alpha}
\end{eqnarray*}
for all balls $B_R \subset \mbb R^n$ of radius $R \geq 1$, which implies 
that the function ${\mathcal H}(x) := \| f \|_{L^1(S)}^{-p} |Ef(x)|^p H(x)$ 
is a weight on $\mbb R^n$ of dimension
\begin{displaymath}
\alpha' = \Big( 1 - \frac{n-1}{2n}p \Big) \alpha + \frac{p\epsilon}{2}
\end{displaymath}
and with
\begin{displaymath} 
A_{\alpha'}({\mathcal H}) 
\lct A_\alpha(H) \| f \|_{L^1(S)}^{-p} \| f \|_{L^2(S)}^p.
\end{displaymath}
Motivated by what we did in the plane, we want to choose a $p \in (0,2]$ 
that will place $\alpha'$ between $(n-1)/2$ and $n/2$ and minimize the 
exponent $q_0$ given by
\begin{displaymath}
q_0 = \frac{4 \alpha'}{n-1} + p 
    = \frac{4 \alpha}{n-1} + \Big( 1 - \frac{2 \alpha}{n} \Big) p +  
      \frac{2p\epsilon}{n-1}.
\end{displaymath}
Since $1-(2\alpha/n) < 0$, $q$ is smallest when $p$ is largest. Also, since
$\epsilon$ can be chosen arbitrarily small,
\begin{displaymath}
\frac{n-1}{2} \leq \alpha' \leq \frac{n}{2} \;\;\; \Longleftarrow \;\;\;
\frac{n-1}{2} \leq \Big( 1 - \frac{n-1}{2n}p \Big) \alpha < \frac{n}{2}.
\end{displaymath}
Therefore,
\begin{displaymath}
p= \frac{2n}{n-1} - \frac{n}{\alpha}.
\end{displaymath}
We note that $p \leq 2$ if $\alpha \leq n(n-1)/2$, which is satisfied 
because $\alpha \leq n$ and $n \geq 3$.

Since $(n-1)/2 \leq \alpha' \leq n/2$, we now have 
$Q(\alpha',2) \leq 4\alpha'/(n-1)$, and so
\begin{displaymath}
\int |Eg(x)|^{q'} {\mathcal H}(x) dx \lct A_{\alpha'}({\mathcal H}) 
\| g \|_{L^2(S)}^{q'}
\end{displaymath}
for all $g \in L^2(S)$ provided $q' > (4\alpha'/(n-1))=q_0-p$. Replacing the
weight ${\mathcal H}$ by $\| f \|_{L^1(S)}^{-p} |Ef|^p H$, plugging $f$ for 
$g$, and choosing $\epsilon$ to be sufficiently small, the last estimate 
becomes
\begin{displaymath}
\int |Ef(x)|^q H(x) dx \lct A_\alpha(H) \| f \|_{L^2(S)}^q
\end{displaymath}
for $q > (2n/(n-1))+2-(n/\alpha)$, proving the inequality on the last line 
of (\ref{main1pf1}).
\end{proof}

\section{Preliminaries for the proofs of Theorems \ref{main2} and 
         \ref{main3}}

Let $M(\mbb R^n)$ be the space of all complex Borel measures on $\mbb R^n$.
Suppose $\mu \in M(\mbb R^n)$ is positive and compactly supported, and 
$0 < \alpha < n$. The $\alpha$-dimensional energy of $\mu$ is defined as
\begin{displaymath}
I_\alpha(\mu) = \int \!\!\! \int \frac{1}{|x-y|^\alpha} d\mu(x) d\mu(y).
\end{displaymath}
The integral $I_\alpha(\mu)$ has the following Fourier representation
\begin{equation}
\label{energy}
I_\alpha(\mu)
= c_\alpha \int |\widehat{\mu}(\xi)|^2 \frac{d\xi}{|\xi|^{n-\alpha}}
= c_\alpha \int_0^\infty \| \widehat{\mu}(R \cdot) \|_{L^2(\mbb S^{n-1})}^2
R^{\alpha-1} dR,
\end{equation}
where $c_\alpha$ is a constant that only depends on $\alpha$ and $n$, and
$\mbb S^{n-1}$ is the unit sphere in $\mbb R^n$.

For positive $\mu \in M(\mbb R^n)$ and $0 < \alpha < n$, we also define
\begin{displaymath}
{\mathcal C}_\alpha(\mu)
= \sup_{x \in \mbb R^n, r > 0} \frac{\mu(B(x,r))}{r^\alpha}.
\end{displaymath}

Let $1 \leq p \leq \infty$ and $p'$ be the exponent conjugate to $p$. We 
want to establish a connection between $L^p(S) \to L^q(\chi_{B(0,R)}Hdx)$ 
restriction estimates and the decay properties of 
$\| \widehat{\mu}(R \cdot) \|_{L^{p'}(S)}$ as $R \to \infty$ for the 
positive measures $\mu \in M(\mbb R^n)$ that are supported in the unit ball 
in $\mbb R^n$ and satisfy $I_\alpha(\mu) < \infty$ or 
${\mathcal C}_\alpha(\mu) < \infty$. 

\begin{prop}
\label{restodecay}
Suppose $1 \leq p \leq \infty$, $q \geq 1$, $0 < \alpha < n$, and we have 
the weighted local restriction estimate
\begin{displaymath}
\int_{B(0,R)} |Ef(x)|^q H(x) dx 
\leq C_\epsilon R^\epsilon A_\alpha(H) \| f \|_{L^p(S)}^q.
\end{displaymath}
Then 
\begin{displaymath}
\| \widehat{\mu}(R \cdot) \|_{L^{p'}(S)} \leq C_\epsilon R^\epsilon 
{\mathcal C}_\alpha(\mu) R^{-\alpha/q} \hspace{0.5in} (R \geq 1)
\end{displaymath}
for all positive measures $\mu \in M(\mbb R^n)$ that are supported in 
$B(0,1)$. Moreover, if $q \geq 2$, then
\begin{displaymath}
\| \widehat{\mu}(R \cdot) \|_{L^{p'}(S)} \leq C_\epsilon R^\epsilon 
\sqrt{I_\alpha(\mu)} R^{-\alpha/q} \hspace{0.5in} (R \geq 1)
\end{displaymath}
for all positive measures $\mu \in M(\mbb R^n)$ that are supported in 
$B(0,1)$.
\end{prop}

Proposition \ref{restodecay} is a standard result, which we state and prove
here for clarity of exposition, as well as for highlighting the difference
between the cases $1 \leq q < 2$ and $q \geq 2$. The proof also reveals 
that the result of the proposition does not extend to the $0 < q < 1$ case, 
which is the main reason why Theorem \ref{main3} is much harder to prove 
than Theorem \ref{main2}.

For the proof of Proposition \ref{restodecay}, we need to borrow the 
following two lemmas from \cite{plms12046} and \cite{tw:csoipaper}.

\begin{alphlemma}[{\cite[Lemma 5.1]{plms12046}}]
\label{bdmubyh}
Suppose $\mu \in M(\mbb R^n)$ is positive and supported in $B(0,1)$,
$0 < \alpha \leq n$, $R \geq 1$, and
\begin{displaymath}
{\mathcal C}_{\alpha,R}(\mu)
= \sup_{x \in \mbb R^n} \sup_{r \geq R^{-1}} \frac{\mu(B(x,r))}{r^\alpha}.
\end{displaymath}
Then there is a weight $H$ (which depends on $R$) of dimension $\alpha$ such
that:
\\
{\rm (i)} $A_\alpha(H) \leq |B(0,1)|$.
\\
{\rm (ii)} To every function $f \in L^1(S)$ there is a function
$g \in L^1(S)$ such that $|g| \leq |f|$ and
\begin{displaymath}
\int |Ef(R x)|^q d\mu(x)
\leq C_q \frac{{\mathcal C}_{\alpha,R}(\mu)}{R^\alpha}
     \int_{B(0,2R)} |Eg(y)|^q H(y) dy
\end{displaymath}
for $q \geq 1$, where $C_q$ is a constant that only depends on $n$ and $q$.
\end{alphlemma}

\begin{alphlemma}[{\cite[Lemma 1.5]{tw:csoipaper}}]
\label{soilemma}
Let $\mu \in M(\mbb R^n)$ be a positive measure with support in $B(0,1)$,
$0 < \alpha < n$, and $R \geq 1$. Then we can decompose $\mu$ as a sum of
$O(1+\log R)$ measures $\mu_j$ so that for each $j$,
\begin{displaymath}
\| \mu_j \| \, {\mathcal C}_{\alpha,R}(\mu_j) \lct I_\alpha(\mu)
\end{displaymath}
with an implicit constant that depends only on $\alpha$ and $n$.
\end{alphlemma}

\begin{proof}[Proof of Proposition \ref{restodecay}]
Let $f \in L^1(S)$, and $g$ be as in (ii) of Lemma \ref{bdmubyh}. Then the 
weighted restriction estimate in the assumption of Proposition 
\ref{restodecay} tells us that
\begin{displaymath}
\int |Ef(R x)|^q d\mu(x) \leq 
C_q \frac{{\mathcal C}_{\alpha,R}(\mu)}{R^\alpha} C_\epsilon (2R)^\epsilon 
A_\alpha(H) \| g \|_{L^p(S)}^q,	
\end{displaymath}
so that
\begin{equation}
\label{formujay}
\int |Ef(R x)|^q d\mu(x) \lct R^\epsilon 
\frac{{\mathcal C}_{\alpha,R}(\mu)}{R^\alpha} \| f \|_{L^p(S)}^q,	
\end{equation}
where we have used the facts that $A_\alpha(H) \leq |B(0,1)|$ and 
$|g| \leq |f|$ provided to us by Lemma \ref{bdmubyh}.	
	
Since $q \geq 1$, we can use H\"{o}lder's inequality to get
\begin{displaymath}
\Big( \int |Ef(R x)| d\mu(x) \Big)^q 
\lct R^\epsilon \| \mu \|^{q-1}
\frac{{\mathcal C}_{\alpha,R}(\mu)}{R^\alpha} \| f \|_{L^p(S)}^q.	
\end{displaymath}
Since $\mu$ is supported in the unit ball, we have 
$\| \mu \| = \mu(B(0,1)) \leq {\mathcal C}_{\alpha,R}(\mu)$, so
$\| \mu \|^{q-1} {\mathcal C}_{\alpha,R}(\mu) \leq 
 {\mathcal C}_{\alpha,R}(\mu)^q$, and so
\begin{displaymath}
\int |Ef(R x)| d\mu(x) \lct R^{\epsilon/q} 
\frac{{\mathcal C}_{\alpha,R}(\mu)}{R^{\alpha/q}} \| f \|_{L^p(S)}.
\end{displaymath}
Since $Ef= \widehat{f d\sigma}$, it follows that
\begin{displaymath}
\Big| \int \widehat{\mu}(R\xi) f(\xi) d\sigma(\xi) \Big| 
\lct R^{\epsilon/q} 
     \frac{{\mathcal C}_{\alpha,R}(\mu)}{R^{\alpha/q}} \| f \|_{L^p(S)}	
\end{displaymath}
for all $f \in L^p(S)$. By duality, this implies that
\begin{displaymath}
\| \widehat{\mu}(R \cdot) \|_{L^{p'}(S)} 
\leq C_\epsilon R^{\epsilon/q} {\mathcal C}_{\alpha,R}(\mu) R^{-\alpha/q}
\leq C_\epsilon R^\epsilon {\mathcal C}_\alpha(\mu) R^{-\alpha/q}
\end{displaymath}
for all $R \geq 1$.

Now suppose $q \geq 2$ and write $\mu= \sum_j \mu_j$ as in Lemma 
\ref{soilemma}. By H\"{o}lder's inequality, we have
\begin{eqnarray*}
\int |Ef(R x)| d\mu_j(x) 
& \leq & \| \mu_j \|^{1-(1/q)} \Big( \int |Ef(R x)|^q d\mu_j(x) \Big)^{1/q}
\\
& = & \| \mu_j \|^{1-(2/q)} 
      \Big( \| \mu_j \| \int |Ef(R x)|^q d\mu_j(x) \Big)^{1/q}.
\end{eqnarray*}
Since $q \geq 2$, we have $\| \mu_j \|^{1-(2/q)} \leq \| \mu \|^{1-(2/q)}$.
Also, by applying (\ref{formujay}) to $\mu_j$ and then using the inequality
$\| \mu_j \| \, {\mathcal C}_{\alpha,R}(\mu_j) \lct I_\alpha(\mu)$ from
Lemma \ref{soilemma}, we have
\begin{displaymath}
\| \mu_j \| \int |Ef(R x)|^q d\mu_j(x) \lct R^\epsilon \| \mu_j \|
\frac{{\mathcal C}_{\alpha,R}(\mu_j)}{R^\alpha} \| f \|_{L^p(S)}^q
\lct R^\epsilon
\frac{I_\alpha(\mu)}{R^\alpha} \| f \|_{L^p(S)}^q.	
\end{displaymath}
Therefore,
\begin{displaymath}
\int |Ef(R x)| d\mu_j(x) \lct \| \mu \|^{1-(2/q)} \Big( R^\epsilon
\frac{I_\alpha(\mu)}{R^\alpha} \| f \|_{L^p(S)}^q \Big)^{1/q}.	
\end{displaymath}
Summing over $j$, this gives
\begin{displaymath}
\int |Ef(R x)| d\mu(x) 
\lct (1+\log R) R^{\epsilon/q} \| \mu \|^{1-(2/q)} I_\alpha(\mu)^{1/q}
	     R^{-\alpha/q} \| f \|_{L^p(S)}.
\end{displaymath}
Since $\mbox{supp} \, \mu \subset B(0,1)$, we have
$\| \mu \|^2 \lct I_\alpha(\mu)$, and the above estimate becomes
\begin{displaymath}
\int |Ef(R x)| d\mu(x) 
\lct (1+\log R) R^{\epsilon/q} I_\alpha(\mu)^{1/2} R^{-\alpha/q} 
     \| f \|_{L^p(S)}.
\end{displaymath}
Therefore,
\begin{displaymath}
\Big| \int \widehat{\mu}(R \xi)| f(\xi) d\sigma(\xi) \Big| 
\lct R^\epsilon I_\alpha(\mu)^{1/2} R^{-\alpha/q} 
    \| f \|_{L^p(S)}
\end{displaymath}
for all $f \in L^p(S)$, and the desired inequality, i.e.
\begin{displaymath}
\| \widehat{\mu}(R \cdot) \|_{L^{p'}(S)} \leq C_\epsilon R^\epsilon 
\sqrt{I_\alpha(\mu)} R^{-\alpha/q}
\end{displaymath}
for all $R \geq 1$, follows from duality.	
\end{proof}

We now need to complement Proposition \ref{restodecay} with some of the 
facts that we currently know about the decay properties of
$\| \widehat{\mu}(R \cdot) \|_{L^1(S)}$ and 
$\| \widehat{\mu}(R \cdot) \|_{L^{p'}(S)}$ when $S$ is the unit sphere. The
first fact is the following basic result in geometric measure theory.

\begin{prop}
\label{basicgmt}
Let $0 < \alpha < n$. Then to every pair $(\beta,b)$ of numbers that 
satisfy $\beta > \alpha /2$ and $b > 0$ there is a number $R \geq 1$ and a 
positive measure $\mu \in M(\mbb R^n)$ with 
{\rm supp}$\, \mu \subset B(0,1)$ such that 
$R^\beta \| \widehat{\mu}(R \cdot) \|_{L^2(\mbb S^{n-1})} 
> b \; {\mathcal C}_\alpha(\mu)$.
\end{prop}

\begin{proof}
Suppose the proposition is not true. Then there is a pair $(\beta,b)$ with
$\beta > \alpha /2$ and $b > 0$ such that
$\| \widehat{\mu}(R \cdot) \|_{L^2(\mbb S^{n-1})} \leq b R^{-\beta}
{\mathcal C}_\alpha(\mu)$ for all
$R \geq 1$ and positive $\mu \in M(\mbb R^n)$ that are supported in 
$B(0,1)$.

We let $\gamma < n$ be a number that lies strictly between $\alpha$ and 
$2\beta$, and $K \! \subset B(0,1)$ be a set of Hausdorff dimension 
strictly between $\alpha$ and $\gamma$. Then $K$ carries a probability 
measure $\mu$ with ${\mathcal C}_\alpha(\mu) < \infty$. By the previous 
paragraph, we have
$\| \widehat{\mu}(R \cdot) \|_{L^2(\mbb S^{n-1})} \lct R^{-\beta}$ for all
$R \geq 1$, so (by (\ref{energy})) $I_\gamma(\mu) < \infty$, and so $K$ 
carries a probability measure $\nu$ such that
${\mathcal C}_\gamma(\nu) < \infty$. This implies that $K$ has Hausdorff
dimension $\geq \gamma$, which is a contradiction.
\end{proof}

The second fact that complements Proposition \ref{restodecay} is due to 
Wolff \cite{tw:csoipaper}:

\begin{alphprop}[{\cite[Lemma 3.1]{tw:csoipaper}}]
\label{soifact}
Let $0 < \alpha < n$. Then to every pair $(\beta,b)$ of numbers that satisfy
$\beta > \alpha/2$ and $b > 0$ there is a number $R \geq 1$ and a positive 
measure $\mu \in M(\mbb R^n)$ with {\rm supp}$\, \mu \subset B(0,1)$ such
that $R^\beta \| \widehat{\mu}(R \cdot) \|_{L^1(\mbb S^{n-1})} 
> b \, \sqrt{I_\alpha(\mu)}$.
\end{alphprop}

\begin{proof}
Let $\psi$ be a non-negative $C^\infty$ function on $\mbb R^n$ that is
supported in the unit ball and satisfies $|\widehat{\psi}| \geq 1$ on the
unit sphere. For $0 < \rho \leq 1$ and $x \in \mbb R^n$, we let 
$\Psi(x)=\rho^{(\alpha/2)-n} \psi(\rho^{-1}x)$, and we define the 
measure $\mu$ by $d\mu=\Psi dx$. Then 
$\widehat{\mu}(\xi)=\rho^{(\alpha/2)} \widehat{\psi}(\rho \, \xi)$ and (by 
(\ref{energy}))
\begin{displaymath}
I_\alpha(\mu) = c_\alpha 
\rho^\alpha \int |\widehat{\psi}(\rho \, \xi)|^2 |\xi|^{\alpha-n} d\xi
= c_\alpha \int |\widehat{\psi}(u)|^2 |u|^{\alpha-n} du \sim 1.
\end{displaymath}

Suppose the proposition is not true. Then there are numbers 
$\beta > \alpha/2$ and $b > 0$ such that 
$\| \widehat{\mu}(R \cdot) \|_{L^1(\mbb S^{n-1})} \leq b R^{-\beta}
I_\alpha(\mu)$ for all $R \geq 1$, so that
\begin{displaymath}
\int_{\mbb S^{n-1}} \big| \widehat{\psi}(\rho R \theta) \big| 
d\sigma(\theta) \lct R^{-\beta} \rho^{-\alpha/2}
\end{displaymath}
for all $0 < \rho \leq 1$ and $R \geq 1$. Taking $R=\rho^{-1}$, we get
\begin{displaymath}
\sigma(\mbb S^{n-1}) \leq 
\int_{\mbb S^{n-1}} \big| \widehat{\psi}(\theta) \big| 
d\sigma(\theta) \lct R^{-\beta} R^{\alpha/2}
\end{displaymath}
for all $R \geq 1$, which implies that $\beta \leq \alpha/2$, which is a
contradiction. 
\end{proof} 

The third fact that we need to complement Proposition \ref{restodecay} is 
the following result of Bennett and Vargas \cite{bv:randomised}.

\begin{alphthm}[{\cite[Corollary 2]{bv:randomised}}]
\label{lonemeans}
Let $1 \leq \alpha < 2$. Then to every pair $(\beta,b)$ of numbers that
satisfy $\beta > \alpha /(\alpha+2)$ and $ b > 0$ there is a positive 
measure $\mu \in M(\mbb R^2)$ with {\rm supp}$\, \mu \subset B(0,1)$ such 
that $R^\beta \| \widehat{\mu}(R \cdot) \|_{L^1(\mbb S^1)} 
> b \, \sqrt{I_\alpha(\mu)}$.
\end{alphthm}

For the interesting proof of Theorem \ref{lonemeans}, we refer the reader to
\cite{bv:randomised}.

\section{Proof of Theorems \ref{main2}}

We are given the estimate
\begin{equation}
\label{givenl2tolq}
\int_{B_R} |Ef(x)|^q H(x) dx 
\leq C_\epsilon R^\epsilon A_\alpha(H) \| f \|_{L^2(S)}^q
\end{equation}
for some $q > 0$, and we need to show that
\begin{equation}
\label{lowbdltwo}
q \geq \left\{ \begin{array}{ll}
2 & \mbox{ if $n \geq 2$ and $0 < \alpha < n$,} \\
(2\alpha+2)/(n-1) & \mbox{ if $n \geq 2$ and $1 < \alpha \leq n$,} \\
4\alpha & \mbox{ if $n = 2$ and $1/2 \leq \alpha \leq 1$.} 
\end{array} \right.
\end{equation} 
In fact, the first line of (\ref{lowbdltwo}) proves the first and third 
lower bounds on $q$ in Theorem \ref{main2}, the second line of
(\ref{lowbdltwo}) proves the second lower bound on $q$ in Theorem 
\ref{main2}, and the third line of (\ref{lowbdltwo}) is identical to the 
fourth lower bound on $q$ in Theorem \ref{main2}.

In proving (\ref{lowbdltwo}), we proceed backwards starting with the 
inequality on its last line.

Suppose $n=2$ and $1/2 \leq \alpha \leq 1$. We let $b=1-\alpha$ and define
\begin{displaymath}
X_b = \{ (x,y) \in \mbb R^2 : x > 0 \mbox{ and } 0 \leq y \leq x^{-b} \}.
\end{displaymath}
Recall from Subsection 3.4 that the characteristic function of $X_b$ is a 
weight on $\mbb R^2$ of dimension $1-b=\alpha$, and 
$A_\alpha(\chi_{X_b}) \lct 1$. So (\ref{givenl2tolq}) implies that
\begin{displaymath}
\int_{X_b \cap B_R} |E f(x)|^q dx \lct R^\epsilon \| f \|_{L^2(S)}^q
\end{displaymath}
for all $R \geq 1$. We now use the same Knapp-example argument that we used
in the Introduction.

To every $R > 1$ there is a function $f_R$ on $S$ such that 
$\| f_R \|_{L^2(S)} \lct R^{-1/4}$ and $|Ef_R| \gct R^{-1/2}$ on the 
rectangle $[0,R] \times [0,\sqrt{R}]$. The intersection of this rectangle 
with $X_b \cap B(0,R)$ contains the rectangle $[0,R] \times [0,R^{-b}]$, so 
$\| Ef_R \|_{L^q(X_b \cap B(0,R))}^q$ $\gct$ $R^{(-q/2)+\alpha}$, and so 
$R^{(-q/2)+\alpha} \leq R^\epsilon R^{-q/4}$. Therefore, $q \geq 4\alpha$.

Moving to the second line of (\ref{lowbdltwo}), we now suppose that 
$n \geq 2$ and $1 < \alpha \leq n$. We let $b=(\alpha-1)/(n-1)$ and define
\begin{displaymath}
\Omega_b = \cup_{l=1}^\infty \mbb R \times [l^{1/b},1+l^{1/b}]^{n-1}
\end{displaymath}
and we observe that $0 < b \leq 1$ and the characteristic function of 
$\Omega_b$ is a weight on $\mbb R^n$ of dimension $1+(n-1)b=\alpha$ and 
with $A_\alpha(\chi_{\Omega_b}) \lct 1$. So (\ref{givenl2tolq}) (applied 
with $H=\chi_{\Omega_b}$) implies that
\begin{displaymath}
\int_{\Omega_b \cap B_R} |E f(x)|^q dx \lct R^\epsilon \| f \|_{L^2(S)}^q
\end{displaymath}
for all $R \geq 1$.

To every $R > 1$, there is a function $f_R$ on $S$ satisfying
$\| f_R \|_{L^2(S)} \lct R^{-(n-1)/4}$ and $|Ef_R| \gct R^{-(n-1)/2}$ on 
$[0,R] \times [l^{1/b},1+l^{1/b}]^{n-1}$ whenever $l^{1/b} \leq \sqrt{R}$. 
Since there are $\sim R^{(n-1)b/2}$ such boxes, we see that 
$\| Ef_R \|_{L^q(\Omega_b \cap B(0,R))} \gct R^m$ with 
\begin{displaymath}
m= -\frac{n-1}{2} + \Big( 1 + \frac{(n-1)b}{2} \Big) \frac{1}{q}.
\end{displaymath}
We have $R^{mq} \lct R^\epsilon R^{-(n-1)q/4}$ for all $R > 1$, so
$m \leq -(n-1)/4$, and it follows that $(n-1)q \geq 2\alpha+2$.

Suppose $n \geq 2$ and $0 < \alpha < n$. We will prove the first inequality 
in (\ref{lowbdltwo}) by contradiction. Assume $q < 2$. Then the estimate 
(\ref{givenl2tolq}) holds with $q$ replaced by an exponent $q_0$ that
satisfies $q < q_0 < 2$ and $q_0 \geq 1$ (and $C_\epsilon$ replaced by
$C_\epsilon \sigma(S)^{(q_0-q)/2}$). When we combine the resulting estimate 
with Proposition \ref{restodecay}, we get the decay estimate
\begin{displaymath}
\| \widehat{\mu}(R \cdot) \|_{L^2(S)} \leq C_\epsilon' R^\epsilon 
{\mathcal C}_\alpha(\mu) R^{-\alpha/q_0} \hspace{0.5in} (R \geq 1)
\end{displaymath}
for all positive measures $\mu \in M(\mbb R^n)$ that are supported in 
$B(0,1)$. Proposition \ref{basicgmt} now implies that 
$\alpha/q_0 \leq \alpha /2$, which implies that $q_0 \geq 2$, which is a 
contradiction.

\section{Proof of Theorem \ref{main3}}

We need to show that
\begin{equation}
\label{lowbdlinfty}
Q_{\rm loc}(\alpha,\infty) \geq \left\{ \begin{array}{ll}
(n-1)/n & \mbox{ if $n \geq 2$ and $0 < \alpha < n-1$,} \\
2\alpha/(n-1) & \mbox{ if $n \geq 2$ and $0 < \alpha \leq n$,} \\
3\alpha & \mbox{ if $n = 2$ and $0 < \alpha \leq 1$,} \\
\alpha + 2 & \mbox{ if $n=2$ and $1 \leq \alpha \leq 2$.} 
\end{array} \right.
\end{equation}
In fact, the first line of (\ref{lowbdlinfty}) proves the first and third 
lower bound on $Q_{\rm loc}(\alpha,\infty)$ in Theorem \ref{main3}, the 
second line of (\ref{lowbdlinfty}) proves the second lower bound on 
$Q_{\rm loc}(\alpha,\infty)$ in Theorem \ref{main3}, the third line of 
(\ref{lowbdlinfty}) proves the fourth lower bound on 
$Q_{\rm loc}(\alpha,\infty)$ in Theorem \ref{main3}, and the last line of
(\ref{lowbdlinfty}) proves the fifth lower bound on 
$Q_{\rm loc}(\alpha,\infty)$ in Theorem \ref{main3}.

In proving (\ref{lowbdlinfty}), we proceed backwards starting with the 
inequality on its last line.

Suppose that $n = 2$ and $1 \leq \alpha \leq 2$, and that we have the 
estimate
\begin{equation}
\label{givenrestll}
\int_{B_R} |Ef(x)|^r H(x) dx 
\leq C_\epsilon R^\epsilon A_\alpha(H) \| f \|_{L^\infty(S)}^r
\end{equation}
for some $r > 0$. We need to prove that $r \geq \alpha+2$. We will do this
by showing that $r < \alpha+2$ leads to a contradiction. 

Suppose (\ref{givenrestll}) holds for some $r < \alpha+2$. We let $q$ be an 
exponent that satisfies $r < q < \alpha+2$ and $q \geq 2$. Then
(\ref{givenrestll}) holds with $r$ replaced by $q$ and $C_\epsilon$ 
replaced by $\sigma(S)^{q-r} C_\epsilon$. Since $q \geq 2$, it follows by 
Proposition \ref{restodecay} that 
\begin{displaymath}
\| \widehat{\mu}(R \cdot) \|_{L^1(S)} \leq C_\epsilon' R^\epsilon 
\sqrt{I_\alpha(\mu)} R^{-\alpha/q} \hspace{0.5in} (R \geq 1)
\end{displaymath}
for all positive measures $\mu \in M(\mbb R^n)$ that are supported in 
$B(0,1)$. By Theorem \ref{lonemeans} it then follows that
$\alpha/q \leq \alpha/(\alpha+2)$, which implies that $q \geq \alpha+2$,
which is a contradiction.

We now move to the inequality before the last in (\ref{lowbdlinfty}). So we
are still in the plane, but now $0 < \alpha \leq 1$. We have just proved
that $Q_{\rm loc}(1,\infty) \geq 3$, so, by Corollary \ref{indconclusion},
we have
\begin{displaymath}
\frac{Q_{\rm loc}(\alpha,\infty)}{\alpha} \geq
\frac{Q_{\rm loc}(1,\infty)}{1} \geq 3,
\end{displaymath}
and so $Q_{\rm loc}(\alpha,\infty) \geq 3 \alpha$.

Suppose that $n \geq 2$ and $0 < \alpha \leq n$. The fact that
$Q_{\rm loc}(n,\infty) \geq 2n/(n-1)$ follows from the fact that the
$|\widehat{\sigma}(\xi)| \sim |\xi|^{-(n-1)/2}$ for large $\xi$. Applying
Corollary \ref{indconclusion} as in the previous paragraph, we obtain the
second inequality in (\ref{lowbdlinfty}). 

The rest of the proof will be concerned with the first inequality in 
(\ref{lowbdlinfty}). 

Suppose that $n \geq 2$ and $0 < \alpha < n-1$, and that we have the 
estimate
\begin{equation}
\label{givenrest}
\int_{B_R} |Ef(x)|^r H(x) dx 
\leq C_\epsilon R^\epsilon A_\alpha(H) \| f \|_{L^\infty(S)}^r
\end{equation}
for some $r > 0$. We need to prove that $r \geq (n-1)/n$.

We apply the Cauchy-Schwarz inequality in (\ref{givenrest}) to get
\begin{eqnarray*}
\int_{B_R} |Ef(x)|^{r/2} H(x) dx 
& \leq & \Big( C_\epsilon R^\epsilon A_\alpha(H) \| f \|_{L^\infty(S)}^r 
         \Big)^{1/2} \Big( \int_{B_R} H(x) dx \Big)^{1/2} \\
& \leq & C_\epsilon^{1/2} A_\alpha(H) \| f \|_{L^\infty(S)}^{r/2} R^\beta
\end{eqnarray*}
for all balls $B_R \subset \mbb R^n$ of radius $R \geq 1$, where 
$\beta= (\alpha+\epsilon)/2$. This means 
${\mathcal H} := \| f \|_{L^1(S)}^{-r/2}|Ef|^{r/2} H$ is a weight of 
dimension $\beta$ with
\begin{displaymath}
A_\beta({\mathcal H}) \leq C_\epsilon^{1/2} A_\alpha(H) 
\| f \|_{L^1(S)}^{-r/2} \| f \|_{L^\infty(S)}^{r/2}.
\end{displaymath}
We have $0 < \alpha/2 < (n-1)/2$. So, from here on, we may assume that 
$\epsilon$ is small enough for us to have $0 < \beta < (n-1)/2$, which will
allow us to apply part (ii) of Proposition \ref{base} with any weight of
dimension $\beta$. 

We let $B_\rho$ be a ball in $\mbb R^n$ of radius $\rho \geq 1$, and we 
apply part (ii) of Proposition \ref{base} with the weight 
$\chi_{B_\rho} {\mathcal H}$ to get 
\begin{displaymath}
\int_{B_\rho} |Ef(x)|^q {\mathcal H}(x) dx 
\leq \bar{C}_q^q A_\beta({\mathcal H})^{1/(n-2\beta)} 
     \| {\mathcal H} \|_{L^2(B_\rho)}^{(n-1-2\beta)/(n-2\beta)} 
     \| f \|_{L^2(S)}^q
\end{displaymath}
for $q > q_0$, where 
\begin{displaymath}
q_0 = \frac{n+1-2\beta}{n-2\beta}.
\end{displaymath} 
We already have the bound on $A_\beta({\mathcal H})$ from the previous 
paragraph. Also, (\ref{givenrest}) tells us that to every $\delta' > 0$ 
there is a constant $C_{\delta'}$ such that
\begin{displaymath}
\| {\mathcal H} \|_{L^2(B_\rho)}^2 
= \| f \|_{L^1(S)}^{-r} \int_{B_\rho} |Ef(x)|^r H(x)^2 dx
\leq \| f \|_{L^1(S)}^{-r}
     C_{\delta'} \rho^{\delta'} A_\alpha(H) \| f \|_{L^\infty(S)}^r,
\end{displaymath}
where we have used the fact that $H^2 \leq H$. So
\begin{eqnarray*}
\lefteqn{\int_{B_\rho} \| f \|_{L^1(S)}^{-r/2} |Ef(x)|^{q+(r/2)} H(x) dx} \\
& \leq & \bar{C}_q^q \Big( C_\epsilon^{1/2} A_\alpha(H) 
         \| f \|_{L^1(S)}^{-r/2} \| f \|_{L^\infty(S)}^{r/2}  
         \Big)^{1/(n-2\beta)} \\
&      & \times \Big( \| f \|_{L^1(S)}^{-r} C_{\delta'} \rho^{\delta'} 
         A_\alpha(H) \| f \|_{L^\infty(S)}^r 
         \Big)^{(n-1-2\beta)/(2(n-2\beta))} \| f \|_{L^2(S)}^q \\
& \leq & C_{q,\epsilon,\delta} \, \rho^\delta 
         A_\alpha(H)^{(n+1-2\beta)/(2(n-2\beta))} \| f \|_{L^1(S)}^{-r/2} 
         \| f \|_{L^\infty(S)}^{r/2} \| f \|_{L^2(S)}^q
\end{eqnarray*}
provided $q > (n+1-2\beta)/(n-2\beta)$, where
$\delta= \delta' (n-1-2\beta)/(2(n-2\beta))$, and so
\begin{equation}
\label{covid}
\int_{B_\rho} |Ef(x)|^{q+(r/2)} H(x) dx
\leq C_{q,\epsilon,\delta} \, \rho^\delta A_\alpha(H)^{q_0/2}
     \| f \|_{L^2(S)}^q \| f \|_{L^\infty(S)}^{r/2}.
\end{equation}

We now let $\mu \in M(\mbb R^n)$ be positive, supported in the unit ball
$B(0,1)$, and satisfies ${\mathcal C}_\alpha(\mu) < \infty$. Since
$q_0=(n+1-2\beta)/(n-2\beta) > 1$, we have $q > 1$, and so we can apply 
Lemma \ref{bdmubyh} (with $q+(r/2)$ replacing $q$) to get a weight $H$ on 
$\mbb R^n$ of dimension $\alpha$ that satisfies
\begin{itemize}
\item $A_\alpha(H) \leq |B(0,1)|$
\item to every function $f \in L^1(S)$ there is a function $g \in L^1(S)$
      such that $|g| \leq |f|$ and
\begin{displaymath}	
\int |Ef(\rho \, x)|^{q+(r/2)} d\mu(x) 
\leq C'' \frac{{\mathcal C}_\alpha(\mu)}{\rho^\alpha}
     \int_{B(0,2\rho)} |Eg(y)|^{q+(r/2)} H(y) dy,
\end{displaymath}
where $C''$ depends on $q$, $r$, and $n$. 
\end{itemize}
Then (\ref{covid}) implies that
\begin{displaymath}	
\int |Ef(\rho \, x)|^{q+(r/2)} d\mu(x) 
\leq C''' \frac{{\mathcal C}_\alpha(\mu)}{\rho^\alpha}
     \rho^\delta \| f \|_{L^2(S)}^q \| f \|_{L^\infty(S)}^{r/2}.
\end{displaymath}
Letting $\gamma=\alpha-\delta$ and $p=q+(r/2)$, and using H\"{o}lder's
inequality, this becomes
\begin{displaymath}	
\int |Ef(\rho \, x)| \, d\mu(x) 
\leq C \frac{1}{\rho^{\gamma/p}}
      \| f \|_{L^2(S)}^{q/p} \| f \|_{L^\infty(S)}^{1-q/p}.
\end{displaymath}
Therefore,
\begin{equation}
\label{twoinfty}	
\Big| \int \widehat{\mu}(\rho \, \xi) f(\xi) d\sigma(\xi) \Big|
\leq C \frac{1}{\rho^{\gamma/p}}
      \| f \|_{L^2(S)}^{q/p} \| f \|_{L^\infty(S)}^{1-(q/p)}.
\end{equation}

We will use (\ref{twoinfty}) to estimate the $\sigma$-measure of the set
\begin{displaymath}
\{ \xi \in S : |\widehat{\mu}(\rho \, \xi)| > \lambda \}
\end{displaymath}
for $0 < \lambda \leq \| \mu \|$. For such $\lambda$ and for 
$l \in \mbb N$, we set
\begin{displaymath}
X_l= X_l(\lambda) = \{ \xi \in S : 2^{l-1} \lambda < 
|\widehat{\mu}(\rho \, \xi)| \leq 2^l \lambda \}.
\end{displaymath}
Clearly,
\begin{displaymath}
\{ \xi \in S : |\widehat{\mu}(\rho \, \xi)| > \lambda \} \subset 
\cup_{l=1}^\infty X_l.
\end{displaymath}
Inserting $\overline{\widehat{\mu}(\rho \, \xi)} \; \chi_{X_l}(\xi)$ for 
$f(\xi)$ in (\ref{twoinfty}), we obtain
\begin{displaymath} 
\Big( \int_{X_l} |\widehat{\mu}(\rho \, \xi)|^2 d\sigma(\xi)
\Big)^{1-(q/(2p))} \leq C \rho^{-\gamma/p} (2^l \lambda)^{1-(q/p)},
\end{displaymath}
which implies that
\begin{displaymath}
\Big( \sigma(X_l) (2^{l-1} \lambda)^2 \Big)^{1-(q/(2p))} 
\leq C \rho^{-\gamma/p} (2^l \lambda)^{1-(q/p)},
\end{displaymath}
which in turn implies that
\begin{displaymath}
\sigma(X_l) \lct \rho^{-2\gamma/(2p-q)} (2^l \lambda)^{-2p/(2p-q)}
\end{displaymath}
for all $l \in \mbb N$. Since $2p-q=q+r$, we have 
$\sum_{l=1}^\infty 2^{-2lp/(2p-q)} \sim 1$, and hence
\begin{displaymath}
\sigma(\{ \xi \in S : |\widehat{\mu}(\rho \, \xi)| > \lambda \})
\leq \sum_{l=1}^\infty \sigma(X_l) 
\lct \rho^{-2\gamma/(2p-q)} \lambda^{-2p/(2p-q)}.
\end{displaymath}
Of course, we also have the trivial bound
\begin{displaymath}
\sigma(\{ \xi \in S : |\widehat{\mu}(R \, \xi)| > \lambda \}) \leq
\sigma(S) \lct 1.
\end{displaymath}

We now let $p_0=2p/(2p-q)$ and use the two bounds we now have on the 
$\sigma$ measure of the set 
$\{ \xi \in S : |\widehat{\mu}(\rho \, \xi)| > \lambda \}$ to see that
\begin{eqnarray*}
\lefteqn{\int_0^{\| \mu \|}
\sigma(\{ \xi \in S : |\widehat{\mu}(\rho \, \xi)| > \lambda \})
\lambda^{p_0-1} d\lambda} \\
& & \lct \int_0^{\rho^{-\gamma/p}} \lambda^{p_0-1} d\lambda
         + \rho^{-\gamma p_0/p} \int_{\rho^{-\gamma/p}}^{\| \mu \|}
           \frac{d\lambda}{\lambda} 
    \; \lct \; \rho^{-\gamma p_0/p} \, \log \rho
\end{eqnarray*}
provided $\rho \geq (\| \mu \|^{-1} + \| \mu \|)^{p/\gamma}$. Thus
\begin{displaymath}
\int |\widehat{\mu}(\rho \, \xi)|^{p_0} d\sigma(\xi)
\lct (\log \rho) \rho^{-\gamma p_0/p}.
\end{displaymath}
Since $p_0 \leq 2$, it follows that
\begin{displaymath}
\int |\widehat{\mu}(\rho \, \xi)|^2 d\sigma(\xi)
\leq \| \mu \|^{2-p_0} \int |\widehat{\mu}(\rho \, \xi)|^{p_0} d\sigma(\xi)
\lct (\log \rho) \rho^{-\gamma p_0/p}.
\end{displaymath}

The inequality we just derived is true for all positive measures 
$\mu \in M(\mbb R^n)$ that are supported in the unit ball and satisfy
${\mathcal C}_\alpha(\mu) < \infty$. So, by Proposition \ref{basicgmt},
$\gamma p_0/p \leq \alpha$. Recalling that $\gamma=\alpha-\delta$, and 
letting $\delta \to 0$, we see that $p_0 \leq p$. 

Replacing $p_0$ by its value in terms of $p$ and $q$, we see that
$2p-q \geq 2$. Replacing $p$ by its value in terms of $q$ and $r$, this
becomes $r \geq 2-q$. Since this is true for every 
$q > (n+1-2\beta)/(n-2\beta)$, it follows that
\begin{displaymath}
r \geq 2 - \frac{n+1-2\beta}{n-2\beta}.
\end{displaymath}
Recalling that $\beta=(\alpha+\epsilon)/2$, and letting $\epsilon \to 0$,
we arrive at
\begin{equation}
\label{lastequation}
r \geq 2 - \frac{n+1-\alpha}{n-\alpha}.
\end{equation}

If $\alpha' \leq \alpha$, then any weight $H$ on $\mbb R^n$ of dimension 
$\alpha'$ is also a weight of dimension $\alpha$. Moreover,
$A_\alpha(H) \leq A_{\alpha'}(H)$. So the given estimate (\ref{givenrest}) 
holds for all weights on $\mbb R^n$ of dimension $\alpha' \leq \alpha$, and 
so we can send $\alpha \to 0$ in (\ref{lastequation}) to get
\begin{displaymath}
r \geq 2 - \frac{n+1}{n} = \frac{n-1}{n},
\end{displaymath}
as promised.

\begin{remark}
\label{lastremark}
During the proof of Theorem \ref{main3}, we used the fact that $q+(r/2) > 1$
to obtain the decay estimate
\begin{displaymath}
\| \widehat{\mu}(\rho \, \cdot) \|_{L^2(S)}
\lct (\log \rho)^{1/2} \rho^{-\gamma p_0/(2p)}
\end{displaymath}
for all positive measures $\mu \in M(\mbb R^n)$ that are supported in the 
unit ball and satisfy ${\mathcal C}_\alpha(\mu) < \infty$. If, for some
reason, we knew that $q+(r/2) \geq 2$, then {\rm (\ref{covid})} (via 
Proposition \ref{restodecay}) would give us the decay estimate
\begin{displaymath}
\| \widehat{\mu}(\rho \, \cdot) \|_{L^1(S)}
\lct \rho^\epsilon \rho^{-\alpha/(q+(r/2))}
\end{displaymath}
for all positive measures $\mu \in M(\mbb R^n)$ that are supported in the 
unit ball and satisfy $I_\alpha(\mu) < \infty$, which would have allowed us 
to use Proposition \ref{soifact} to conclude that 
\begin{displaymath}
\frac{\alpha}{q+(r/2)} \leq \frac{\alpha}{2},
\end{displaymath}
i.e.\ $r/2 \geq 2-q$. Proceeding as we did in the last part of the proof of 
Theorem \ref{main3}, we would have arrived at $r \geq 2(n-1)/n$.
\end{remark}

\end{document}